\documentclass[10pt]{amsart}
\usepackage{a4}

\usepackage{xcolor}

\usepackage{amsmath}
\usepackage{amssymb}

\usepackage{longtable}
\usepackage{booktabs}

\usepackage{enumerate}
\usepackage{enumitem}

\usepackage[all]{xy}
\usepackage{tikz}	
\usepackage{tikz-3dplot}
\usepackage{tikz-cd}
\usepackage{graphicx}

\theoremstyle{plain}

\newtheorem{theorem}{Theorem}[section]
\newtheorem{lemma}[theorem]{Lemma}
\newtheorem{proposition}[theorem]{Proposition}

\theoremstyle{definition}

\newtheorem{setting}[theorem]{Setting}
\newtheorem{construction}[theorem]{Construction}
\newtheorem{remark}[theorem]{Remark}

\theoremstyle{remark}

\numberwithin{equation}{section}

\usetikzlibrary{patterns,matrix,arrows,shapes}
\usetikzlibrary{decorations.pathreplacing}

\newcommand\CC{\mathbb{C}}
\newcommand\QQ{\mathbb{Q}}
\newcommand\ZZ{\mathbb{Z}}
\newcommand\PP{\mathbb{P}}
\newcommand\TT{\mathbb{T}}

\newcommand\AAA{\mathcal{A}}
\newcommand\RRR{\mathcal{R}}
\newcommand\OOO{\mathcal{O}}
\newcommand\nnn{\mathfrak n}
\newcommand\KKK{\mathcal{K}}

\newcommand\bangle[1]{\langle #1 \rangle}
\newcommand\git{/\!\!/}
\newcommand\Cl{\mathrm{Cl}}

\begin{document}
\title[On canonical Fano intrinsic quadrics]%
{On canonical Fano intrinsic quadrics}


\author[C.~Hische]{Christoff Hische} 
\address{Mathematisches Institut, Eberhard Karls Universit\"at T\"ubingen, Auf der Morgenstelle 10, 72076 T\"ubingen, Germany}
\email{hische@math.uni-tuebingen.de}

\begin{abstract}
We classify all $\QQ$-factorial Fano intrinsic quadrics of dimension three and Picard number one having at most canonical singularities.
\end{abstract}

\maketitle

\section{Introduction}
\noindent
This article contributes to the classification of Fano 3-folds, 
i.e.\ normal projective varieties of dimension three with an 
ample anticanonical divisor. 
For the smooth Fano 3-folds, the work of 
Iskovskih~\cite{Isk1977, Isk1978} and Mori/Mukai~\cite{MorMuk1981} provides a detailed picture. 
The singular case, in contrast, is widely open
in general.
Toric Fano 3-folds with at most canonical singularities, 
have been completely classified by Kasprzyk in~\cite{Kas2006,Kas2010}.

In the present article, we consider {\em intrinsic quadrics}.
These are normal, projective varieties with a Cox ring 
defined by a single quadratic relation~$q$, see~\cite{BerHau2003}.
Intrinsic quadrics have been used i.a.\ by Bourqui in~\cite{Bou2011} 
as a testing ground for Manin's conjecture. 
Moreover, \cite{FahHau2020} gives concrete descriptions 
of all smooth intrinsic quadrics in the Picard numbers
one and two.
Every Fano intrinsic quadric $X$ is completely determined 
by its Cox ring 
$$
\RRR(X) 
\ = \ 
\bigoplus_{\Cl(X)} \RRR(X)_w 
\ = \ 
\CC[T_1,\ldots,T_r]/\bangle{q}.
$$ 
If $X$ is $\QQ$-factorial and of Picard number one, then
we regain $X$ from its Cox ring as follows:
The quasitorus $H$ with character group 
$\mathbb{X}(H) \cong \Cl(X)$ acts 
diagonally on $\CC^r$ via the characters 
corresponding to the degrees $
w_1,\ldots,w_r \in \Cl(X)$
of the generators $T_1,\ldots,T_r$.
Our variety $X$ equals the good quotient 
$(V(q) \setminus \{0\}) \git H$.

The description of $X$ via its Cox ring allows us to explicitly compute certain invariants of $X$ i.a. its anticanonical self intersection number $-\KKK_X^3$ and its Fano index, i.e.\ the largest integer $q(X)$ such that $-\KKK_X = q(X)\cdot w$ holds for some $w~\in~\Cl(X)$.

\begin{theorem}\label{theorem:classification}
Every $\QQ$-factorial Fano intrinsic quadric of 
dimension three and Picard number one with at
most canonical singularities is isomorphic to 
precisely one of the varieties $X$  in the list 
below, specified by its Cox ring 
$\CC[T_1,\ldots,T_r]/\bangle{q}$  
and the matrix $Q$ having the
$\Cl(X)$-degrees $w_i$ of the generators $T_i$ 
as its columns. 
{\setlength{\tabcolsep}{5pt}\setlength{\arraycolsep}{2pt}
\begin{longtable}{ccccccc}
No.
&
$\RRR(X)$
&
$\Cl(X)$
&
$Q=[w_1,\ldots,w_r]$
&
$-\KKK_X$
&
$q(X)$
&
$-\KKK_X^3$
\\
\toprule
1
&
$
\frac{\CC[T_{1},T_{2},T_{3},T_{4},T_{5}]}{\bangle{T_{1}T_{2}+T_3T_4+T_5^2}}
$
&
$\ZZ$
&
{\tiny
$
\left[\begin{array}{ccccc}
1&1&1&1&1
\end{array}\right]
$
}
&
{\tiny
$
\left[
\begin{array}{c}
3
\end{array}
\right]
$
}
&
$3$
&
$54$
\\
\midrule
2
&
$
\frac{\CC[T_{1},T_{2},T_{3},T_{4},T_{5}]}{\bangle{T_{1}T_{2}+T_3T_4+T_5^2}}
$
&
$\ZZ$
&
{\tiny
$
\left[\begin{array}{ccccc}
2&2&1&3&2
\end{array}\right]
$
}
&
{\tiny
$
\left[
\begin{array}{c}
6
\end{array}
\right]
$
}
&
$6$
&
$36$
\\
\midrule
3
&
$
\frac{\CC[T_{1},T_{2},T_{3},T_{4},T_{5}]}{\bangle{T_{1}T_{2}+T_3T_4+T_5^2}}
$
&
$\ZZ$
&
{\tiny
$
\left[\begin{array}{ccccc}
1&3&1&3&2
\end{array}\right]
$
}
&
{\tiny
$
\left[
\begin{array}{c}
6
\end{array}
\right]
$
}&
$6$
&
$48$
\\
\midrule
4
&
$
\frac{\CC[T_{1},T_{2},T_{3},T_{4},T_{5}]}{\bangle{T_{1}T_{2}+T_3T_4+T_5^2}}
$
&
$\ZZ$
&
{\tiny
$
\left[\begin{array}{ccccc}
2&4&1&5&3
\end{array}\right]
$
}
&
{\tiny
$
\left[
\begin{array}{c}
9
\end{array}
\right]
$
}
&
$9$
&
$\frac{729}{20}$
\\
\midrule
5
&
$
\frac{\CC[T_{1},T_{2},T_{3},T_{4},T_{5}]}{\bangle{T_{1}T_{2}+T_3T_4+T_5^2}}
$
&
$\ZZ$
&
{\tiny
$
\left[\begin{array}{ccccc}
2&6&3&5&4
\end{array}\right]
$
}
&
{\tiny
$
\left[
\begin{array}{c}
12
\end{array}
\right]
$
}
&
$12$
&
$\frac{96}{5}$
\\
\midrule
6
&
$
\frac{\CC[T_{1},T_{2},T_{3},T_{4},T_{5}]}{\bangle{T_{1}T_{2}+T_3T_4+T_5^2}}
$
&
$\ZZ$
&
{\tiny
$
\left[\begin{array}{ccccc}
3&5&1&7&4
\end{array}\right]
$
}
&
{\tiny
$
\left[
\begin{array}{c}
12
\end{array}
\right]
$
}
&
$12$
&
$\frac{1152}{35}$
\\
\midrule
7
&
$
\frac{\CC[T_{1},T_{2},T_{3},T_{4},T_{5}]}{\bangle{T_{1}T_{2}+T_3T_4+T_5^2}}
$
&
$\ZZ$
&
{\tiny
$
\left[\begin{array}{ccccc}
3&7&2&8&5
\end{array}\right]
$
}
&
{\tiny
$
\left[
\begin{array}{c}
15
\end{array}
\right]
$
}
&
$15$
&
$\frac{1125}{56}$
\\
\midrule
8
&
$
\frac{\CC[T_{1},T_{2},T_{3},T_{4},S_{1}]}{\bangle{T_{1}T_{2}+T_3^2+T_4^2}}
$
&
$\ZZ\times \ZZ_2$
&
{\tiny
$
\left[\begin{array}{ccccc}
4&2&3&3&2
\\
\bar{1}&\bar{1}&\bar{0}&\bar{1}&\bar{0}
\end{array}\right]
$
}
&
{\tiny
$
\left[
\begin{array}{c}
8
\\
\bar{1}
\end{array}
\right]
$
}
&
$1$
&
$\frac{32}{3}$
\\
\midrule
9
&
$
\frac{\CC[T_{1},T_{2},T_{3},T_{4},S_{1}]}{\bangle{T_{1}T_{2}+T_3^2+T_4^2}}
$
&
$\ZZ\times \ZZ_2$
&
{\tiny
$
\left[\begin{array}{ccccc}
6&4&5&5&2
\\
\bar{1}&\bar{1}&\bar{0}&\bar{1}&\bar{0}
\end{array}\right]
$
}
&
{\tiny
$
\left[
\begin{array}{c}
12
\\
\bar{1}
\end{array}
\right]
$
}
&
$3$
&
$\frac{36}{5}$
\\
\midrule
10
&
$
\frac{\CC[T_{1},T_{2},T_{3},T_{4},S_{1}]}{\bangle{T_{1}T_{2}+T_3^2+T_4^2}}
$
&
$\ZZ\times \ZZ_2$
&
{\tiny
$
\left[\begin{array}{ccccc}
4&2&3&3&6
\\
\bar{1}&\bar{1}&\bar{0}&\bar{1}&\bar{0}
\end{array}\right]
$
}
&
{\tiny
$
\left[
\begin{array}{c}
12
\\
\bar{1}
\end{array}
\right]
$
}
&
$3$
&
$12$
\\
\midrule
11
&
$
\frac{\CC[T_{1},T_{2},T_{3},T_{4},T_{5}]}{\bangle{T_{1}T_{2}+T_3T_4+T_5^2}}
$
&
$\ZZ\times \ZZ_2$
&
{\tiny
$
\left[\begin{array}{ccccc}
1&3&1&3&2\\
\bar{1}&\bar{1}&\bar{0}&\bar{0}&\bar{1}
\end{array}\right]
$
}
&
{\tiny
$
\left[
\begin{array}{c}
6
\\
\bar{1}
\end{array}
\right]
$
}
&
$3$
&
$24$
\\
\midrule
12
&
$
\frac{\CC[T_{1},T_{2},T_{3},T_{4},S_{1}]}{\bangle{T_{1}T_{2}+T_3^2+T_4^2}}
$
&
$\ZZ\times \ZZ_2$
&
{\tiny
$
\left[\begin{array}{ccccc}
1&1&1&1&1
\\
\bar{1}&\bar{1}&\bar{1}&\bar{0}&\bar{0}
\end{array}\right]
$
}
&
{\tiny
$
\left[
\begin{array}{c}
3
\\
\bar{1}
\end{array}
\right]
$
}
&
$3$
&
$27$
\\
\midrule
13
&
$
\frac{\CC[T_{1},T_{2},T_{3},T_{4},T_{5}]}{\bangle{T_{1}T_{2}+T_3T_4+T_5^2}}
$
&
$\ZZ\times \ZZ_2$
&
{\tiny
$
\left[\begin{array}{ccccc}
1&1&1&1&1\\
\bar{1}&\bar{1}&\bar{0}&\bar{0}&\bar{0}
\end{array}\right]
$
}
&
{\tiny
$
\left[
\begin{array}{c}
3
\\
\bar{0}
\end{array}
\right]
$
}
&
$3$
&
$54$
\\
\midrule
14
&
$
\frac{\CC[T_{1},T_{2},T_{3},T_{4},S_{1}]}{\bangle{T_{1}T_{2}+T_3^2+T_4^2}}
$
&
$\ZZ\times \ZZ_2$
&
{\tiny
$
\left[\begin{array}{ccccc}
1&1&1&1&2
\\
\bar{0}&\bar{0}&\bar{1}&\bar{0}&\bar{1}
\end{array}\right]
$
}
&
{\tiny
$
\left[
\begin{array}{c}
4
\\
\bar{0}
\end{array}
\right]
$
}
&
$4$
&
$32$
\\
\midrule
15
&
$
\frac{\CC[T_{1},T_{2},T_{3},T_{4},S_{1}]}{\bangle{T_{1}T_{2}+T_3^2+T_4^2}}
$
&
$\ZZ\times \ZZ_2$
&
{\tiny
$
\left[\begin{array}{ccccc}
3&1&2&2&1
\\
\bar{1}&\bar{1}&\bar{1}&\bar{0}&\bar{0}
\end{array}\right]
$
}
&
{\tiny
$
\left[
\begin{array}{c}
5
\\
\bar{1}
\end{array}
\right]
$
}
&
$5$
&
$\frac{125}{6}$
\\
\midrule
16
&
$
\frac{\CC[T_{1},T_{2},T_{3},T_{4},S_{1}]}{\bangle{T_{1}T_{2}+T_3^2+T_4^2}}
$
&
$\ZZ\times \ZZ_2$
&
{\tiny
$
\left[\begin{array}{ccccc}
1&3&2&2&2
\\
\bar{0}&\bar{0}&\bar{0}&\bar{1}&\bar{1}
\end{array}\right]
$
}
&
{\tiny
$
\left[
\begin{array}{c}
6
\\
\bar{0}
\end{array}
\right]
$
}
&
$6$
&
$18$
\\
\midrule
17
&
$
\frac{\CC[T_{1},T_{2},T_{3},T_{4},T_{5}]}{\bangle{T_{1}T_{2}+T_3T_4+T_5^2}}
$
&
$\ZZ\times \ZZ_2$
&
{\tiny
$
\left[\begin{array}{ccccc}
2&2&1&3&2\\
\bar{1}&\bar{1}&\bar{0}&\bar{0}&\bar{0}
\end{array}\right]
$
}
&
{\tiny
$
\left[
\begin{array}{c}
6
\\
\bar{0}
\end{array}
\right]
$
}
&
$6$
&
$18$
\\
\midrule
18
&
$
\frac{\CC[T_{1},T_{2},T_{3},T_{4},T_{5}]}{\bangle{T_{1}T_{2}+T_3T_4+T_5^2}}
$
&
$\ZZ\times \ZZ_2$
&
{\tiny
$
\left[\begin{array}{ccccc}
1&3&1&3&2\\
\bar{1}&\bar{1}&\bar{0}&\bar{0}&\bar{0}
\end{array}\right]
$
}
&
{\tiny
$
\left[
\begin{array}{c}
6
\\
\bar{0}
\end{array}
\right]
$
}
&
$6$
&
$24$
\\
\midrule
19
&
$
\frac{\CC[T_{1},T_{2},T_{3},T_{4},S_{1}]}{\bangle{T_{1}T_{2}+T_3^2+T_4^2}}
$
&
$\ZZ\times \ZZ_2$
&
{\tiny
$
\left[\begin{array}{ccccc}
4&2&3&3&1
\\
\bar{1}&\bar{1}&\bar{1}&\bar{0}&\bar{0}
\end{array}\right]
$
}
&
{\tiny
$
\left[
\begin{array}{c}
7
\\
\bar{1}
\end{array}
\right]
$
}
&
$7$
&
$\frac{343}{48}$
\\
\midrule
20
&
$
\frac{\CC[T_{1},T_{2},T_{3},T_{4},S_{1}]}{\bangle{T_{1}T_{2}+T_3^2+T_4^2}}
$
&
$\ZZ\times \ZZ_2$
&
{\tiny
$
\left[\begin{array}{ccccc}
4&2&3&3&2
\\
\bar{1}&\bar{1}&\bar{1}&\bar{0}&\bar{1}
\end{array}\right]
$
}
&
{\tiny
$
\left[
\begin{array}{c}
8
\\
\bar{0}
\end{array}
\right]
$
}
&
$8$
&
$\frac{32}{3}$
\\
\midrule
21
&
$
\frac{\CC[T_{1},T_{2},T_{3},T_{4},S_{1}]}{\bangle{T_{1}T_{2}+T_3^2+T_4^2}}
$
&
$\ZZ\times \ZZ_2$
&
{\tiny
$
\left[\begin{array}{ccccc}
4&2&3&3&2
\\
\bar{0}&\bar{0}&\bar{1}&\bar{0}&\bar{1}
\end{array}\right]
$
}
&
{\tiny
$
\left[
\begin{array}{c}
8
\\
\bar{0}
\end{array}
\right]
$
}
&
$8$
&
$\frac{32}{3}$
\\
\midrule
22
&
$
\frac{\CC[T_{1},T_{2},T_{3},T_{4},S_{1}]}{\bangle{T_{1}T_{2}+T_3^2+T_4^2}}
$
&
$\ZZ\times \ZZ_2$
&
{\tiny
$
\left[\begin{array}{ccccc}
5&1&3&3&2
\\
\bar{0}&\bar{0}&\bar{1}&\bar{0}&\bar{1}
\end{array}\right]
$
}
&
{\tiny
$
\left[
\begin{array}{c}
8
\\
\bar{0}
\end{array}
\right]
$
}
&
$8$
&
$\frac{256}{15}$
\\
\midrule
23
&
$
\frac{\CC[T_{1},T_{2},T_{3},T_{4},S_{1}]}{\bangle{T_{1}T_{2}+T_3^2+T_4^2}}
$
&
$\ZZ\times \ZZ_2$
&
{\tiny
$
\left[\begin{array}{ccccc}
1&3&2&2&4
\\
\bar{0}&\bar{0}&\bar{1}&\bar{0}&\bar{1}
\end{array}\right]
$
}
&
{\tiny
$
\left[
\begin{array}{c}
8
\\
\bar{0}
\end{array}
\right]
$
}
&
$8$
&
$\frac{64}{3}$
\\
\midrule
24
&
$
\frac{\CC[T_{1},T_{2},T_{3},T_{4},S_{1}]}{\bangle{T_{1}T_{2}+T_3^2+T_4^2}}
$
&
$\ZZ\times \ZZ_2$
&
{\tiny
$
\left[\begin{array}{ccccc}
5&3&4&4&1
\\
\bar{1}&\bar{1}&\bar{1}&\bar{0}&\bar{0}
\end{array}\right]
$
}
&
{\tiny
$
\left[
\begin{array}{c}
9
\\
\bar{1}
\end{array}
\right]
$
}
&
$9$
&
$\frac{243}{20}$
\\
\midrule
25
&
$
\frac{\CC[T_{1},T_{2},T_{3},T_{4},S_{1}]}{\bangle{T_{1}T_{2}+T_3^2+T_4^2}}
$
&
$\ZZ\times \ZZ_2$
&
{\tiny
$
\left[\begin{array}{ccccc}
3&5&4&4&2
\\
\bar{1}&\bar{1}&\bar{1}&\bar{0}&\bar{1}
\end{array}\right]
$
}
&
{\tiny
$
\left[
\begin{array}{c}
10
\\
\bar{0}
\end{array}
\right]
$
}
&
$10$
&
$\frac{25}{3}$
\\
\midrule
26
&
$
\frac{\CC[T_{1},T_{2},T_{3},T_{4},S_{1}]}{\bangle{T_{1}T_{2}+T_3^2+T_4^2}}
$
&
$\ZZ\times \ZZ_2$
&
{\tiny
$
\left[\begin{array}{ccccc}
5&1&3&3&4
\\
\bar{1}&\bar{1}&\bar{1}&\bar{0}&\bar{1}
\end{array}\right]
$
}
&
{\tiny
$
\left[
\begin{array}{c}
10
\\
\bar{0}
\end{array}
\right]
$
}
&
$10$
&
$\frac{50}{3}$
\\
\midrule
27
&
$
\frac{\CC[T_{1},T_{2},T_{3},T_{4},S_{1}]}{\bangle{T_{1}T_{2}+T_3^2+T_4^2}}
$
&
$\ZZ\times \ZZ_2$
&
{\tiny
$
\left[\begin{array}{ccccc}
6&4&5&5&2
\\
\bar{1}&\bar{1}&\bar{0}&\bar{1}&\bar{1}
\end{array}\right]
$
}
&
{\tiny
$
\left[
\begin{array}{c}
12
\\
\bar{0}
\end{array}
\right]
$
}
&
$12$
&
$\frac{36}{5}$
\\
\midrule
28
&
$
\frac{\CC[T_{1},T_{2},T_{3},T_{4},S_{1}]}{\bangle{T_{1}T_{2}+T_3^2+T_4^2}}
$
&
$\ZZ\times \ZZ_2$
&
{\tiny
$
\left[\begin{array}{ccccc}
6&4&5&5&2
\\
\bar{0}&\bar{0}&\bar{0}&\bar{1}&\bar{1}
\end{array}\right]
$
}
&
{\tiny
$
\left[
\begin{array}{c}
12
\\
\bar{0}
\end{array}
\right]
$
}
&
$12$
&
$\frac{36}{5}$
\\
\midrule
29
&
$
\frac{\CC[T_{1},T_{2},T_{3},T_{4},S_{1}]}{\bangle{T_{1}T_{2}+T_3^2+T_4^2}}
$
&
$\ZZ\times \ZZ_2$
&
{\tiny
$
\left[\begin{array}{ccccc}
7&3&5&5&2
\\
\bar{1}&\bar{1}&\bar{1}&\bar{0}&\bar{1}
\end{array}\right]
$
}
&
{\tiny
$
\left[
\begin{array}{c}
12
\\
\bar{0}
\end{array}
\right]
$
}
&
$12$
&
$\frac{288}{35}$
\\
\midrule
30
&
$
\frac{\CC[T_{1},T_{2},T_{3},T_{4},S_{1}]}{\bangle{T_{1}T_{2}+T_3^2+T_4^2}}
$
&
$\ZZ\times \ZZ_2$
&
{\tiny
$
\left[\begin{array}{ccccc}
2&4&3&3&6
\\
\bar{0}&\bar{0}&\bar{1}&\bar{0}&\bar{1}
\end{array}\right]
$
}
&
{\tiny
$
\left[
\begin{array}{c}
12
\\
\bar{0}
\end{array}
\right]
$
}
&
$12$
&
$12$
\\
\midrule
31
&
$
\frac{\CC[T_{1},T_{2},T_{3},T_{4},S_{1}]}{\bangle{T_{1}T_{2}+T_3^2+T_4^2}}
$
&
$\ZZ\times \ZZ_2$
&
{\tiny
$
\left[\begin{array}{ccccc}
2&4&3&3&6
\\
\bar{1}&\bar{1}&\bar{1}&\bar{0}&\bar{1}
\end{array}\right]
$
}
&
{\tiny
$
\left[
\begin{array}{c}
12
\\
\bar{0}
\end{array}
\right]
$
}
&
$12$
&
$12$
\\
\midrule
32
&
$
\frac{\CC[T_{1},T_{2},T_{3},T_{4},S_{1}]}{\bangle{T_{1}T_{2}+T_3^2+T_4^2}}
$
&
$\ZZ\times \ZZ_2$
&
{\tiny
$
\left[\begin{array}{ccccc}
7&3&5&5&4
\\
\bar{0}&\bar{0}&\bar{1}&\bar{0}&\bar{1}
\end{array}\right]
$
}
&
{\tiny
$
\left[
\begin{array}{c}
14
\\
\bar{0}
\end{array}
\right]
$
}
&
$14$
&
$\frac{98}{15}$
\\
\midrule
33
&
$
\frac{\CC[T_{1},T_{2},T_{3},T_{4},T_{5}]}{\bangle{T_{1}T_{2}+T_3T_4+T_5^2}}
$
&
$\ZZ\times \ZZ_3$
&
{\tiny
$
\left[\begin{array}{ccccc}
1&1&1&1&1\\
\bar{1}&\bar{2}&\bar{0}&\bar{0}&\bar{0}
\end{array}\right]
$
}
&
{\tiny
$
\left[
\begin{array}{c}
3
\\
\bar{0}
\end{array}
\right]
$
}
&
$3$
&
$18$
\\
\midrule
34
&
$
\frac{\CC[T_{1},T_{2},T_{3},T_{4},T_{5}]}{\bangle{T_{1}T_{2}+T_3T_4+T_5^2}}
$
&
$\ZZ\times \ZZ_3$
&
{\tiny
$
\left[\begin{array}{ccccc}
1&1&1&1&1\\
\bar{1}&\bar{2}&\bar{1}&\bar{2}&\bar{0}
\end{array}\right]
$
}
&
{\tiny
$
\left[
\begin{array}{c}
3
\\
\bar{0}
\end{array}
\right]
$
}
&
$3$
&
$18$
\\
\midrule
35
&
$
\frac{\CC[T_{1},T_{2},T_{3},T_{4},T_{5}]}{\bangle{T_{1}T_{2}+T_3T_4+T_5^2}}
$
&
$\ZZ\times \ZZ_3$
&
{\tiny
$
\left[\begin{array}{ccccc}
2&2&1&3&2\\
\bar{0}&\bar{2}&\bar{0}&\bar{2}&\bar{1}
\end{array}\right]
$
}
&
{\tiny
$
\left[
\begin{array}{c}
6
\\
\bar{0}
\end{array}
\right]
$
}
&
$6$
&
$12$
\\
\midrule
36
&
$
\frac{\CC[T_{1},T_{2},T_{3},T_{4},T_{5}]}{\bangle{T_{1}T_{2}+T_3T_4+T_5^2}}
$
&
$\ZZ\times \ZZ_3$
&
{\tiny
$
\left[\begin{array}{ccccc}
2&2&1&3&2\\
\bar{0}&\bar{2}&\bar{2}&\bar{0}&\bar{1}
\end{array}\right]
$
}
&
{\tiny
$
\left[
\begin{array}{c}
6
\\
\bar{0}
\end{array}
\right]
$
}
&
$6$
&
$12$
\\
\midrule
37
&
$
\frac{\CC[T_{1},T_{2},T_{3},T_{4},T_{5}]}{\bangle{T_{1}T_{2}+T_3T_4+T_5^2}}
$
&
$\ZZ\times \ZZ_3$
&
{\tiny
$
\left[\begin{array}{ccccc}
1&3&1&3&2\\
\bar{0}&\bar{2}&\bar{1}&\bar{1}&\bar{1}
\end{array}\right]
$
}
&
{\tiny
$
\left[
\begin{array}{c}
6
\\
\bar{0}
\end{array}
\right]
$
}
&
$6$
&
$16$
\\
\midrule
38
&
$
\frac{\CC[T_{1},T_{2},T_{3},T_{4},S_{1}]}{\bangle{T_{1}T_{2}+T_3^2+T_4^2}}
$
&
$\ZZ\times \ZZ_4$
&
{\tiny
$
\left[\begin{array}{ccccc}
1&1&1&1&2
\\
\bar{1}&\bar{3}&\bar{2}&\bar{0}&\bar{3}
\end{array}\right]
$
}
&
{\tiny
$
\left[
\begin{array}{c}
4
\\
\bar{1}
\end{array}
\right]
$
}
&
$1$
&
$19$
\\
\midrule
39
&
$
\frac{\CC[T_{1},T_{2},T_{3},T_{4},S_{1}]}{\bangle{T_{1}T_{2}+T_3^2+T_4^2}}
$
&
$\ZZ\times \ZZ_4$
&
{\tiny
$
\left[\begin{array}{ccccc}
2&4&3&3&2
\\
\bar{3}&\bar{1}&\bar{0}&\bar{2}&\bar{0}
\end{array}\right]
$
}
&
{\tiny
$
\left[
\begin{array}{c}
8
\\
\bar{2}
\end{array}
\right]
$
}
&
$2$
&
$\frac{16}{3}$
\\
\midrule
40
&
$
\frac{\CC[T_{1},T_{2},T_{3},T_{4},S_{1}]}{\bangle{T_{1}T_{2}+T_3^2+T_4^2}}
$
&
$\ZZ\times \ZZ_4$
&
{\tiny
$
\left[\begin{array}{ccccc}
1&1&1&1&2
\\
\bar{2}&\bar{0}&\bar{1}&\bar{3}&\bar{2}
\end{array}\right]
$
}
&
{\tiny
$
\left[
\begin{array}{c}
4
\\
\bar{2}
\end{array}
\right]
$
}
&
$2$
&
$16$
\\
\midrule
41
&
$
\frac{\CC[T_{1},T_{2},T_{3},T_{4},T_{5}]}{\bangle{T_{1}T_{2}+T_3T_4+T_5^2}}
$
&
$\ZZ\times \ZZ_4$
&
{\tiny
$
\left[\begin{array}{ccccc}
1&1&1&1&1\\
\bar{2}&\bar{2}&\bar{1}&\bar{3}&\bar{0}
\end{array}\right]
$
}
&
{\tiny
$
\left[
\begin{array}{c}
3
\\
\bar{0}
\end{array}
\right]
$
}
&
$3$
&
$\frac{27}{2}$
\\
\midrule
42
&
$
\frac{\CC[T_{1},T_{2},T_{3},T_{4},S_{1}]}{\bangle{T_{1}T_{2}+T_3^2+T_4^2}}
$
&
$\ZZ\times \ZZ_4$
&
{\tiny
$
\left[\begin{array}{ccccc}
1&1&1&1&1
\\
\bar{2}&\bar{0}&\bar{1}&\bar{3}&\bar{3}
\end{array}\right]
$
}
&
{\tiny
$
\left[
\begin{array}{c}
3
\\
\bar{3}
\end{array}
\right]
$
}
&
$3$
&
$\frac{27}{2}$
\\
\midrule
43
&
$
\frac{\CC[T_{1},T_{2},T_{3},T_{4},S_{1}]}{\bangle{T_{1}T_{2}+T_3^2+T_4^2}}
$
&
$\ZZ\times \ZZ_4$
&
{\tiny
$
\left[\begin{array}{ccccc}
1&1&1&1&1
\\
\bar{0}&\bar{2}&\bar{1}&\bar{3}&\bar{0}
\end{array}\right]
$
}
&
{\tiny
$
\left[
\begin{array}{c}
3
\\
\bar{0}
\end{array}
\right]
$
}
&
$3$
&
$\frac{27}{2}$
\\
\midrule
44
&
$
\frac{\CC[T_{1},T_{2},T_{3},T_{4},S_{1}]}{\bangle{T_{1}T_{2}+T_3^2+T_4^2}}
$
&
$\ZZ\times \ZZ_4$
&
{\tiny
$
\left[\begin{array}{ccccc}
1&1&1&1&2
\\
\bar{2}&\bar{0}&\bar{1}&\bar{3}&\bar{0}
\end{array}\right]
$
}
&
{\tiny
$
\left[
\begin{array}{c}
4
\\
\bar{0}
\end{array}
\right]
$
}
&
$4$
&
$16$
\\
\midrule
45
&
$
\frac{\CC[T_{1},T_{2},T_{3},T_{4},S_{1}]}{\bangle{T_{1}T_{2}+T_3^2+T_4^2}}
$
&
$\ZZ\times \ZZ_4$
&
{\tiny
$
\left[\begin{array}{ccccc}
3&1&2&2&1
\\
\bar{1}&\bar{1}&\bar{1}&\bar{3}&\bar{0}
\end{array}\right]
$
}
&
{\tiny
$
\left[
\begin{array}{c}
5
\\
\bar{0}
\end{array}
\right]
$
}
&
$5$
&
$\frac{49}{6}$
\\
\midrule
46
&
$
\frac{\CC[T_{1},T_{2},T_{3},T_{4},S_{1}]}{\bangle{T_{1}T_{2}+T_3^2+T_4^2}}
$
&
$\ZZ\times \ZZ_4$
&
{\tiny
$
\left[\begin{array}{ccccc}
1&3&2&2&2
\\
\bar{1}&\bar{1}&\bar{3}&\bar{1}&\bar{0}
\end{array}\right]
$
}
&
{\tiny
$
\left[
\begin{array}{c}
6
\\
\bar{0}
\end{array}
\right]
$
}
&
$6$
&
$9$
\\
\midrule
47
&
$
\frac{\CC[T_{1},T_{2},T_{3},T_{4},S_{1}]}{\bangle{T_{1}T_{2}+T_3^2+T_4^2}}
$
&
$\ZZ\times \ZZ_4$
&
{\tiny
$
\left[\begin{array}{ccccc}
3&1&2&2&2
\\
\bar{2}&\bar{0}&\bar{1}&\bar{3}&\bar{0}
\end{array}\right]
$
}
&
{\tiny
$
\left[
\begin{array}{c}
6
\\
\bar{0}
\end{array}
\right]
$
}
&
$6$
&
$9$
\\
\midrule
48
&
$
\frac{\CC[T_{1},T_{2},T_{3},T_{4},S_{1}]}{\bangle{T_{1}T_{2}+T_3^2+T_4^2}}
$
&
$\ZZ\times \ZZ_4$
&
{\tiny
$
\left[\begin{array}{ccccc}
2&4&3&3&2
\\
\bar{1}&\bar{3}&\bar{2}&\bar{0}&\bar{2}
\end{array}\right]
$
}
&
{\tiny
$
\left[
\begin{array}{c}
8
\\
\bar{0}
\end{array}
\right]
$
}
&
$8$
&
$\frac{16}{3}$
\\
\midrule
49
&
$
\frac{\CC[T_{1},T_{2},T_{3},T_{4},T_{5}]}{\bangle{T_{1}T_{2}+T_3T_4+T_5^2}}
$
&
$\ZZ\times \ZZ_5$
&
{\tiny
$
\left[\begin{array}{ccccc}
1&1&1&1&1\\
\bar{2}&\bar{3}&\bar{1}&\bar{4}&\bar{0}
\end{array}\right]
$
}
&
{\tiny
$
\left[
\begin{array}{c}
3
\\
\bar{0}
\end{array}
\right]
$
}
&
$3$
&
$\frac{54}{5}$
\\
\midrule
50
&
$
\frac{\CC[T_{1},T_{2},T_{3},T_{4},T_{5}]}{\bangle{T_{1}T_{2}+T_3T_4+T_5^2}}
$
&
$\ZZ\times \ZZ_6$
&
{\tiny
$
\left[\begin{array}{ccccc}
1&1&1&1&1\\
\bar{1}&\bar{5}&\bar{2}&\bar{4}&\bar{0}
\end{array}\right]
$
}
&
{\tiny
$
\left[
\begin{array}{c}
3
\\
\bar{0}
\end{array}
\right]
$
}
&
$3$
&
$9$
\\
\midrule
51
&
$
\frac{\CC[T_{1},T_{2},T_{3},T_{4},T_{5}]}{\bangle{T_{1}T_{2}+T_3T_4+T_5^2}}
$
&
$\ZZ\times \ZZ_6$
&
{\tiny
$
\left[\begin{array}{ccccc}
1&1&1&1&1\\
\bar{2}&\bar{4}&\bar{3}&\bar{3}&\bar{0}
\end{array}\right]
$
}
&
{\tiny
$
\left[
\begin{array}{c}
3
\\
\bar{0}
\end{array}
\right]
$
}
&
$3$
&
$9$
\\
\midrule
52
&
$
\frac{\CC[T_{1},T_{2},T_{3},T_{4},S_{1}]}{\bangle{T_{1}T_{2}+T_3^2+T_4^2}}
$
&
$\ZZ\times \ZZ_6$
&
{\tiny
$
\left[\begin{array}{ccccc}
1&1&1&1&1
\\
\bar{4}&\bar{0}&\bar{2}&\bar{5}&\bar{5}
\end{array}\right]
$
}
&
{\tiny
$
\left[
\begin{array}{c}
3
\\
\bar{0}
\end{array}
\right]
$
}
&
$3$
&
$9$
\\
\midrule
53
&
$
\frac{\CC[T_{1},T_{2},T_{3},T_{4},S_{1}]}{\bangle{T_{1}T_{2}+T_3^2+T_4^2}}
$
&
$\ZZ\times \ZZ_6$
&
{\tiny
$
\left[\begin{array}{ccccc}
1&1&1&1&2
\\
\bar{4}&\bar{0}&\bar{5}&\bar{2}&\bar{1}
\end{array}\right]
$
}
&
{\tiny
$
\left[
\begin{array}{c}
4
\\
\bar{2}
\end{array}
\right]
$
}
&
$4$
&
$\frac{32}{3}$
\\
\midrule
54
&
$
\frac{\CC[T_{1},T_{2},T_{3},T_{4},S_{1}]}{\bangle{T_{1}T_{2}+T_3^2+T_4^2}}
$
&
$\ZZ\times \ZZ_8$
&
{\tiny
$
\left[\begin{array}{ccccc}
1&1&1&1&1
\\
\bar{2}&\bar{0}&\bar{5}&\bar{1}&\bar{3}
\end{array}\right]
$
}
&
{\tiny
$
\left[
\begin{array}{c}
3
\\
\bar{1}
\end{array}
\right]
$
}
&
$3$
&
$\frac{27}{4}$
\\
\midrule
55
&
$
\frac{\CC[T_{1},T_{2},T_{3},T_{4},S_{1}]}{\bangle{T_{1}T_{2}+T_3^2+T_4^2}}
$
&
$\ZZ\times \ZZ_8$
&
{\tiny
$
\left[\begin{array}{ccccc}
1&1&1&1&2
\\
\bar{2}&\bar{0}&\bar{5}&\bar{1}&\bar{2}
\end{array}\right]
$
}
&
{\tiny
$
\left[
\begin{array}{c}
4
\\
\bar{0}
\end{array}
\right]
$
}
&
$4$
&
$8$
\\
\midrule
56
&
$
\frac{\CC[T_{1},T_{2},T_{3},T_{4},S_{1}]}{\bangle{T_{1}T_{2}+T_3^2+T_4^2}}
$
&
$\ZZ\times \ZZ_8$
&
{\tiny
$
\left[\begin{array}{ccccc}
1&1&1&1&2
\\
\bar{2}&\bar{0}&\bar{5}&\bar{1}&\bar{6}
\end{array}\right]
$
}
&
{\tiny
$
\left[
\begin{array}{c}
4
\\
\bar{4}
\end{array}
\right]
$
}
&
$4$
&
$8$
\\
\midrule
57
&
$
\frac{\CC[T_{1},T_{2},T_{3},T_{4},T_{5}]}{\bangle{T_{1}T_{2}+T_3T_4+T_5^2}}
$
&
$\ZZ\times \ZZ_9$
&
{\tiny
$
\left[\begin{array}{ccccc}
1&1&1&1&1\\
\bar{4}&\bar{5}&\bar{3}&\bar{6}&\bar{0}
\end{array}\right]
$
}
&
{\tiny
$
\left[
\begin{array}{c}
3
\\
\bar{0}
\end{array}
\right]
$
}
&
$3$
&
$6$
\\
\midrule
58
&
$
\frac{\CC[T_{1},T_{2},T_{3},T_{4},S_{1}]}{\bangle{T_{1}T_{2}+T_3^2+T_4^2}}
$
&
$\ZZ\times \ZZ_{12}$
&
{\tiny
$
\left[\begin{array}{ccccc}
1&1&1&1&1
\\
\bar{2}&\bar{0}&\bar{7}&\bar{1}&\bar{4}
\end{array}\right]
$
}
&
{\tiny
$
\left[
\begin{array}{c}
3
\\
\bar{0}
\end{array}
\right]
$
}
&
$3$
&
$\frac{9}{2}$
\\
\midrule
59
&
$
\frac{\CC[T_{1},T_{2},T_{3},S_{1},S_{2}]}{\bangle{T_{1}^2+T_{2}^2+T_{3}^2}}
$
&
$\ZZ\times(\ZZ_2)^2$
&
{\tiny
$
\left[\begin{array}{ccccc}
2&2&2&1&1
\\
\bar{1}&\bar{1}&\bar{0}&\bar{0}&\bar{0}
\\
\bar{0}&\bar{1}&\bar{1}&\bar{1}&\bar{0}
\end{array}\right]
$
}
&
{\tiny
$
\left[
\begin{array}{c}
4
\\
\bar{0}
\\
\bar{1}
\end{array}
\right]
$
}
&
$1$
&
$\frac{19}{2}$
\\
\midrule
60
&
$
\frac{\CC[T_{1},T_{2},T_{3},S_{1},S_{2}]}{\bangle{T_{1}^2+T_{2}^2+T_{3}^2}}
$
&
$\ZZ\times(\ZZ_2)^2$
&
{\tiny
$
\left[\begin{array}{ccccc}
1&1&1&1&2
\\
\bar{1}&\bar{1}&\bar{0}&\bar{0}&\bar{1}
\\
\bar{0}&\bar{1}&\bar{1}&\bar{0}&\bar{1}
\end{array}\right]
$
}
&
{\tiny
$
\left[
\begin{array}{c}
4
\\
\bar{1}
\\
\bar{1}
\end{array}
\right]
$
}
&
$1$
&
$16$
\\
\midrule
61
&
$
\frac{\CC[T_{1},T_{2},T_{3},S_{1},S_{2}]}{\bangle{T_{1}^2+T_{2}^2+T_{3}^2}}
$
&
$\ZZ\times(\ZZ_2)^2$
&
{\tiny
$
\left[\begin{array}{ccccc}
2&2&2&1&3
\\
\bar{1}&\bar{1}&\bar{0}&\bar{0}&\bar{1}
\\
\bar{0}&\bar{1}&\bar{1}&\bar{0}&\bar{1}
\end{array}\right]
$
}
&
{\tiny
$
\left[
\begin{array}{c}
6
\\
\bar{1}
\\
\bar{1}
\end{array}
\right]
$
}
&
$3$
&
$\frac{45}{4}$
\\
\midrule
62
&
$
\frac{\CC[T_{1},T_{2},T_{3},S_{1},S_{2}]}{\bangle{T_{1}^2+T_{2}^2+T_{3}^2}}
$
&
$\ZZ\times(\ZZ_2)^2$
&
{\tiny
$
\left[\begin{array}{ccccc}
1&1&1&1&1
\\
\bar{1}&\bar{1}&\bar{0}&\bar{0}&\bar{0}
\\
\bar{1}&\bar{0}&\bar{0}&\bar{1}&\bar{0}
\end{array}\right]
$
}
&
{\tiny
$
\left[
\begin{array}{c}
3
\\
\bar{0}
\\
\bar{0}
\end{array}
\right]
$
}
&
$3$
&
$\frac{27}{2}$
\\
\midrule
63
&
$
\frac{\CC[T_{1},T_{2},T_{3},T_{4},S_{1}]}{\bangle{T_{1}T_{2}+T_3^2+T_4^2}}
$
&
$\ZZ\times (\ZZ_{2})^2$
&
{\tiny
$
\left[\begin{array}{ccccc}
1&1&1&1&1
\\
\bar{0}&\bar{0}&\bar{1}&\bar{1}&\bar{0}
\\
\bar{1}&\bar{1}&\bar{1}&\bar{0}&\bar{0}
\end{array}\right]
$
}
&
{\tiny
$
\left[
\begin{array}{c}
3
\\
\bar{0}
\\
\bar{1}
\end{array}
\right]
$
}
&
$3$
&
$\frac{27}{2}$
\\
\midrule
64
&
$
\frac{\CC[T_1,\ldots,T_5]}{\bangle{T_1T_2+T_3^2+T_4^2+T_5^2}}
$
&
$\ZZ\times(\ZZ_2)^2$&
{\tiny
$\left[\begin{array}{ccccc}
     1&1&1&1&1  
     \\
     \bar 0&\bar 0&\bar 1&\bar 1&\bar 0
     \\
     \bar 1&\bar 1&\bar 0&\bar 1&\bar 0
\end{array}\right]$
}&
{\tiny
$
\left[
\begin{array}{c}
3
\\
\bar 0
\\
\bar 1
\end{array}
\right]
$
}
&
$3$
&
$\frac{27}{2}$
\\
\midrule
65
&
$
\frac{\CC[T_{1},T_{2},T_{3},T_{4},S_{1}]}{\bangle{T_{1}T_{2}+T_3^2+T_4^2}}
$
&
$\ZZ\times (\ZZ_{2})^2$
&
{\tiny
$
\left[\begin{array}{ccccc}
1&1&1&1&2
\\
\bar{0}&\bar{0}&\bar{1}&\bar{1}&\bar{0}
\\
\bar{0}&\bar{0}&\bar{0}&\bar{1}&\bar{1}
\end{array}\right]
$
}
&
{\tiny
$
\left[
\begin{array}{c}
4
\\
\bar{0}
\\
\bar{0}
\end{array}
\right]
$
}
&
$4$
&
$16$
\\
\midrule
66
&
$
\frac{\CC[T_{1},T_{2},T_{3},S_{1},S_{2}]}{\bangle{T_{1}^2+T_{2}^2+T_{3}^2}}
$
&
$\ZZ\times(\ZZ_2)^2$
&
{\tiny
$
\left[\begin{array}{ccccc}
3&3&3&1&2
\\
\bar{0}&\bar{0}&\bar{1}&\bar{0}&\bar{1}
\\
\bar{1}&\bar{0}&\bar{0}&\bar{0}&\bar{1}
\end{array}\right]
$
}
&
{\tiny
$
\left[
\begin{array}{c}
6
\\
\bar{0}
\\
\bar{0}
\end{array}
\right]
$
}
&
$6$
&
$6$
\\
\midrule
67
&
$
\frac{\CC[T_{1},T_{2},T_{3},S_{1},S_{2}]}{\bangle{T_{1}^2+T_{2}^2+T_{3}^2}}
$
&
$\ZZ\times(\ZZ_2)^2$
&
{\tiny
$
\left[\begin{array}{ccccc}
2&2&2&1&3
\\
\bar{1}&\bar{1}&\bar{0}&\bar{0}&\bar{0}
\\
\bar{1}&\bar{0}&\bar{0}&\bar{0}&\bar{1}
\end{array}\right]
$
}
&
{\tiny
$
\left[
\begin{array}{c}
6
\\
\bar{0}
\\
\bar{0}
\end{array}
\right]
$
}
&
$6$
&
$9$
\\
\midrule
68
&
$
\frac{\CC[T_{1},T_{2},T_{3},T_{4},S_{1}]}{\bangle{T_{1}T_{2}+T_3^2+T_4^2}}
$
&
$\ZZ\times (\ZZ_{2})^2$
&
{\tiny
$
\left[\begin{array}{ccccc}
3&1&2&2&2
\\
\bar{0}&\bar{0}&\bar{1}&\bar{1}&\bar{0}
\\
\bar{0}&\bar{0}&\bar{0}&\bar{1}&\bar{1}
\end{array}\right]
$
}
&
{\tiny
$
\left[
\begin{array}{c}
6
\\
\bar{0}
\\
\bar{0}
\end{array}
\right]
$
}
&
$6$
&
$9$
\\
\midrule
69
&
$
\frac{\CC[T_1,\ldots,T_5]}{\bangle{T_1T_2+T_3^2+T_4^2+T_5^2}}
$
&
$\ZZ\times(\ZZ_2)^2$&
{\tiny
$\left[\begin{array}{ccccc}
     1&3&2&2&2  
     \\
     \bar 1&\bar 1&\bar 1&\bar 1&\bar 0
     \\
     \bar 0 &\bar 0&\bar 1&\bar 0&\bar 1
\end{array}\right]$
}&
{\tiny
$
\left[
\begin{array}{c}
6
\\
\bar0
\\
\bar0
\end{array}
\right]
$
}
&
$6$
&
$9$
\\
\midrule
70
&
$
\frac{\CC[T_{1},T_{2},T_{3},S_{1},S_{2}]}{\bangle{T_{1}^2+T_{2}^2+T_{3}^2}}
$
&
$\ZZ\times(\ZZ_2)^2$
&
{\tiny
$
\left[\begin{array}{ccccc}
1&1&1&3&2
\\
\bar{1}&\bar{1}&\bar{0}&\bar{1}&\bar{1}
\\
\bar{1}&\bar{0}&\bar{0}&\bar{0}&\bar{1}
\end{array}\right]
$
}
&
{\tiny
$
\left[
\begin{array}{c}
6
\\
\bar{0}
\\
\bar{0}
\end{array}
\right]
$
}
&
$6$
&
$18$
\\
\midrule
71
&
$
\frac{\CC[T_{1},T_{2},T_{3},T_{4},S_{1}]}{\bangle{T_{1}T_{2}+T_3^2+T_4^2}}
$
&
$\ZZ\times (\ZZ_{2})^2$
&
{\tiny
$
\left[\begin{array}{ccccc}
2&4&3&3&2
\\
\bar{1}&\bar{1}&\bar{0}&\bar{0}&\bar{0}
\\
\bar{1}&\bar{1}&\bar{1}&\bar{0}&\bar{1}
\end{array}\right]
$
}
&
{\tiny
$
\left[
\begin{array}{c}
8
\\
\bar{0}
\\
\bar{0}
\end{array}
\right]
$
}
&
$8$
&
$\frac{16}{3}$
\\
\midrule
72
&
$
\frac{\CC[T_{1},T_{2},T_{3},S_{1},S_{2}]}{\bangle{T_{1}^2+T_{2}^2+T_{3}^2}}
$
&
$\ZZ\times\ZZ_2\times\ZZ_4$
&
{\tiny
$
\left[\begin{array}{ccccc}
2&2&2&1&1
\\
\bar{1}&\bar{0}&\bar{1}&\bar{0}&\bar{0}
\\
\bar{1}&\bar{3}&\bar{3}&\bar{1}&\bar{0}
\end{array}\right]
$
}
&
{\tiny
$
\left[
\begin{array}{c}
4
\\
\bar{0}
\\
\bar{2}
\end{array}
\right]
$
}
&
$2$
&
$4$
\\
\midrule
73
&
$
\frac{\CC[T_{1},T_{2},T_{3},S_{1},S_{2}]}{\bangle{T_{1}^2+T_{2}^2+T_{3}^2}}
$
&
$\ZZ\times\ZZ_2\times\ZZ_4$
&
{\tiny
$
\left[\begin{array}{ccccc}
1&1&1&2&1
\\
\bar{0}&\bar{1}&\bar{0}&\bar{1}&\bar{0}
\\
\bar{1}&\bar{3}&\bar{3}&\bar{1}&\bar{0}
\end{array}\right]
$
}
&
{\tiny
$
\left[
\begin{array}{c}
4
\\
\bar{0}
\\
\bar{2}
\end{array}
\right]
$
}
&
$2$
&
$8$
\\
\midrule
74
&
$
\frac{\CC[T_{1},T_{2},T_{3},T_{4},S_{1}]}{\bangle{T_{1}T_{2}+T_3^2+T_4^2}}
$
&
$\ZZ\times \ZZ_{2}\times\ZZ_4$
&
{\tiny
$
\left[\begin{array}{ccccc}
1&1&1&1&2
\\
\bar{0}&\bar{0}&\bar{0}&\bar{1}&\bar{1}
\\
\bar{2}&\bar{0}&\bar{1}&\bar{3}&\bar{0}
\end{array}\right]
$
}
&
{\tiny
$
\left[
\begin{array}{c}
4
\\
\bar{0}
\\
\bar{0}
\end{array}
\right]
$
}
&
$4$
&
$8$
\\
\midrule
75
&
$
\frac{\CC[T_{1},T_{2},T_{3},S_{1},S_{2}]}{\bangle{T_{1}^2+T_{2}^2+T_{3}^2}}
$
&
$\ZZ\times\ZZ_2\times\ZZ_6$
&
{\tiny
$
\left[\begin{array}{ccccc}
1&1&1&1&1
\\
\bar{1}&\bar{0}&\bar{1}&\bar{0}&\bar{0}
\\
\bar{4}&\bar{1}&\bar{1}&\bar{5}&\bar{0}
\end{array}\right]
$
}
&
{\tiny
$
\left[
\begin{array}{c}
3
\\
\bar{0}
\\
\bar{3}
\end{array}
\right]
$
}
&
$3$
&
$\frac{9}{2}$
\\
\midrule
76
&
$
\frac{\CC[T_{1},T_{2},T_{3},T_{4},S_{1}]}{\bangle{T_{1}T_{2}+T_3^2+T_4^2}}
$
&
$\ZZ\times \ZZ_{2}\times\ZZ_6$
&
{\tiny
$
\left[\begin{array}{ccccc}
1&1&1&1&1
\\
\bar{0}&\bar{0}&\bar{1}&\bar{0}&\bar{1}
\\
\bar{2}&\bar{0}&\bar{4}&\bar{1}&\bar{1}
\end{array}\right]
$
}
&
{\tiny
$
\left[
\begin{array}{c}
3
\\
\bar{0}
\\
\bar{0}
\end{array}
\right]
$
}
&
$3$
&
$\frac{9}{2}$
\\
\midrule
77
&
$
\frac{\CC[T_1,\ldots,T_5]}
{\bangle{T_1T_2+T_3^2+T_4^2+T_5^2}}
$
&
$\ZZ\times\ZZ_2\times\ZZ_6$&
{\tiny
$\left[\begin{array}{ccccc}
     1&1&1&1&1  
     \\
     \bar 1&\bar 1&\bar 1&\bar 0&\bar 0
     \\
     \bar 2&\bar 4&\bar 3&\bar 3&\bar 0
\end{array}\right]$
}&
{\tiny
$
\left[
\begin{array}{c}
3
\\
1
\\
0
\end{array}
\right]
$
}
&
$3$
&
$\frac{9}{2}$
\\
\midrule
78
&
$
\frac{\CC[T_{1},T_{2},T_{3},T_{4},T_{5}]}{\bangle{T_{1}T_{2}+T_3T_4+T_5^2}}
$
&
$\ZZ\times (\ZZ_3)^2$
&
{\tiny
$
\left[\begin{array}{ccccc}
1&1&1&1&1\\
\bar{1}&\bar{2}&\bar{1}&\bar{2}&\bar{0}\\
\bar{2}&\bar{1}&\bar{1}&\bar{2}&\bar{0}
\end{array}\right]
$
}
&
{\tiny
$
\left[
\begin{array}{c}
3
\\
\bar{0}
\\
\bar{0}
\end{array}
\right]
$
}
&
$3$
&
$6$
\\
\midrule
79
&
{
$
\frac{\CC[T_1,\ldots,T_4,S_1]}{
     \bangle{T_1^2+T_2^2+T_3^2+T_4^2}}$
}
&
$\ZZ\times(\ZZ_2)^3$&
{\tiny
$\left[\begin{array}{ccccc}
     1&1&1&1&2\\
    \bar 1&\bar 1&\bar 1&\bar 0&\bar 1\\
     \bar 0&\bar 0&\bar 1&\bar 0&\bar 1\\
     \bar 1&\bar 0&\bar 0&\bar 0&\bar 1
\end{array}\right]$
}&
{\tiny
$
\left[
\begin{array}{c}
4\\
\bar 0\\
\bar 0\\
\bar 0
\end{array}
\right]
$
}
&
$4$
&
$8$
\\
\bottomrule
\end{longtable}
}
\end{theorem}

To prove this result, we make use of the so called anticanonical complex as firstly introduced in~\cite{BecHauHugNic2016}
for Fano varieties with a torus action of complexity one, i.e. the general torus orbit is of codimension one. 
There, the authors have classified all $\QQ$-factorial Fano 3-folds with Picard number one having at most terminal singularities and admitting a torus action of complexity one. 
Note that in our list all varieties defined by a trinomial quadric admit a torus action of complexity one. 
In particular, varieties Nos.~1, 4, 19 and~49 appear in the classification list of~\cite{BecHauHugNic2016}:
variety No.~1 is smooth, the others have terminal singularities.
In~\cite{HisWro2018} the anticanonical complex has been made accessible for a broader class of varieties, i.a. for the intrinsic quadrics. 
There, all $\QQ$-factorial Fano intrinsic quadrics of dimension three having at most canonical singularities and a torus action of complexity two have been classified. 
These show up as Nos.~64, 69, 77 and~79 in our classification list.

\tableofcontents

\section{Background on intrinsic quadrics}
\noindent
In this section we recall the basic facts about intrinsic quadrics from~\cite{FahHau2020} and adapt the methods developed in~\cite{HauHisWro2019,HisWro2018} to prove our main result in the subsequent section.
Our main tool is the {\em Cox ring} $\RRR(X)$, which can be assigned to any normal projective variety $X$ with finitely generated divisor class group $\Cl(X)$
$$
\RRR(X)=\bigoplus\limits_{[D]\in\Cl(X)}\Gamma(X,\OOO_X(D)).
$$ 
We refer to~\cite{ArzDerHauLaf2015} for a precise definition and background on Cox rings.

An {\em intrinsic quadric} is a normal projective variety $X$ with finitely generated divisor class group $\Cl(X)$ and finitely generated Cox ring $\RRR(X)$ admitting homogeneous generators $f_1,\ldots,f_r$ such that the ideal of relations is generated by a single, purely quadratic relation $q$. 
In particular, we have a graded isomorphism 
$$
\RRR(X)\cong \CC[f_1,\ldots,f_r]/\langle q\rangle.
$$

\begin{remark}
Let $X$ be an intrinsic quadric. 
Then, due to~\cite[Prop. 2.5]{FahHau2020}, there is a graded isomorphism
\begin{equation}\label{equation:normalform}
\RRR(X)
\cong
\CC[T_{1},\ldots,T_{n},S_1,\ldots,S_m]
/
\langle
q_{s,t}
\rangle
,
\end{equation}
where $q_{s,t}:=T_1T_2+\ldots+T_{s-1}T_{s}+T_{s+1}^2+\ldots+T_{s+t}^2$.
The polynomial $q_{s,t}$ is called a {\em standard quadric}
and the representation of $\RRR(X)$ in (\ref{equation:normalform}) is called the {\em homogeneous normal form}.
\end{remark}

The homogeneous normal form enables us to work in the flexible language introduced in~\cite{HauHisWro2019}.
We adapt the basic constructions presented there to intrinsic quadrics and recall the major results.

\begin{construction}\label{construction:RnP_0}
Fix integers $r,m \geq 0$ and $n_0, \ldots, n_r > 0$,  such that $2\geq n_0\geq\ldots\geq n_r\geq 1$ holds.
Set $\mathfrak{n}:=(n_0,\ldots,n_r)$, $n := n_0 + \ldots + n_r$ and define an integral $r\times(n+m)$ matrix $P_0$ built up from tuples $l_0,\ldots,l_r$ as follows:
$$
P_{0}
\ := \
\left[
\begin{array}{ccccccc}
-l_{0} & l_{1} &  & 0 & 0  &  \ldots & 0
\\
\vdots & \vdots & \ddots & \vdots & \vdots &  & \vdots
\\
-l_{0} & 0 &  & l_{r} & 0  &  \ldots & 0
\end{array}
\right],\qquad l_i :=
\left\{
\begin{array}{cl}
(1,1)&\text{if }n_i=2\text{ holds},\\
(2)&\text{else}.
\end{array}
\right.
$$
We will write $\CC[T_{ij},S_k]$ for the polynomial ring in the variables $T_{ij}$ and $S_k$, where $0\leq i\leq r,1\leq j\leq n_i$ and $1\leq k\leq m$ holds.
The $l_i$ define a polynomial
$$q=T_{0}^{l_0}+\ldots+T_r^{l_r},\qquad\text{with }T_i^{l_i}:=T_{i1}^{l_{i1}}\cdots T_{in_i}^{l_{in_i}}\in\CC[T_{ij},S_k].$$
Now, let $e_{ij} \in \ZZ^{n}$ 
and $e_k \in \ZZ^{m}$ denote the 
canonical basis vectors
and consider the projection 
$$
Q_0 \colon \ZZ^{n+m} 
\ \rightarrow \ 
K_0 := \ZZ^{n+m} / \mathrm{im}(P_0^*)
$$ 
onto the factor group
by the row lattice of $P_0$.
We define the {\em $K_0$-graded $\CC$-algebra}
$$ 
R(\mathfrak{n},P_0)
\ := \ 
\CC[T_{ij},S_k] / \langle{q}\rangle,
$$
$$
\deg(T_{ij}) :=  Q_0(e_{ij}),
\qquad
\deg(S_k) :=  Q_0(e_k).
$$
\end{construction}

\begin{remark}
Let $R:=R(\nnn,P_0)$ be a $K_0$-graded $\CC$-algebra as in Construction~\ref{construction:RnP_0}.
Then $R$ is integral and normal if $r\geq 2$ holds.
Moreover, the $K_0$-grading is the finest possible grading leaving the variables $T_{ij}$ and $S_k$ and the relation $q$ homogeneous and the defining relation $q$ is a standard quadric.
\end{remark}

\begin{construction}\label{construction:RnP}
Let $R(\nnn,P_0)$ be a $K_0$-graded $\CC$-algebra as in Construction~\ref{construction:RnP_0}.
Choose any integral $s\times (n+m)$ matrix $D$ with $r+s\leq n+m$, such that the stack matrix
$$P:=\left[\begin{array}{c}
P_0\\
D
\end{array}
\right]$$
has pairwise different and primitive columns generating $\QQ^{r+s}$ as a cone.
Now, similar to Construction~\ref{construction:RnP_0}, consider the factor group $K:=\ZZ^{n+m}/\mathrm{im}(P^*)$ and the projection $Q\colon\ZZ^{n+m}\rightarrow K$.
Then, we define the {\em $K$-graded $\CC$-algebra}
$$ 
R(\mathfrak{n},P)
\ := \ 
\CC[T_{ij},S_k] / \langle{q}\rangle,
$$
$$
\deg(T_{ij}) :=  Q(e_{ij}),
\qquad
\deg(S_k) :=  Q(e_k).
$$
\end{construction}

\begin{remark}
Let $R(\nnn,P)$ be a $K$-graded $\CC$-algebra as in Construction~\ref{construction:RnP}. 
Then the natural homomorphism $K_0\mapsto K, [v]\mapsto[v]$ defines a downgrading from the $K_0$-graded $\CC$-algebra $R(\nnn,P_0)$ to the $K$-graded $\CC$-algebra $R(\nnn,P)$.
\end{remark}

\begin{proposition}
Let $X$ be an intrinsic quadric. 
Then the $\Cl(X)$-graded Cox ring $\RRR(X)$ is isomorphic to a $K$-graded $\CC$-algebra $R(\nnn,P)$ as in Construction~\ref{construction:RnP}.
\end{proposition}

We will now use the rings $R(\nnn,P)$ to construct intrinsic quadrics, suitably embedded inside toric varieties; we refer to~\cite{CoxLitSch2011} for background on toric geometry.
For this, let $R:=R(\nnn,P)$ be a $K$-graded $\CC$-algebra from Construction~\ref{construction:RnP} and denote by $\gamma$ the positive orthant $\QQ^{n+m}_{\geq 0}$.
For any face $\gamma_0\preceq \gamma$, we denote by $\gamma_0^*$ its {\em complementary face}, i.e. $\gamma_0^*:=\mathrm{cone}(e_i;\ e_i\not \in \gamma_0)\preceq\gamma$. Moreover, for a homomorphism of finitely generated abelian groups $A\colon K\rightarrow K'$ we denote its unique extension to the $\QQ$ vector spaces $K_\QQ:=K\otimes_\ZZ \QQ$, resp. $K'_\QQ$ as well with $A\colon K_\QQ\rightarrow K'_\QQ$. Finally, we define a polyhedral cone 
$$\mathrm{Mov}(R):=\bigcap\limits_{\gamma_0\preceq\gamma\text{ facet}} Q(\gamma_0)\subseteq K_\QQ.$$

\begin{construction}\label{construction:IQ}
Consider an integral $K$-graded $\CC$-algebra $R:=R(\nnn,P)$ as in Construction~\ref{construction:RnP}. 
Then the $K$-grading on the polynomial ring $\CC[T_{ij},S_k]$ 
defines an action of the quasitorus $H:=\mathrm{Spec}\ \CC[K]$ on $\overline{Z}:=\CC^{n+m}$ that leaves $\overline{X}:=V(q)\subseteq\overline{Z}$ invariant.
Now, choose any element $u$ inside the relative interior $\mathrm{Mov}(R)^\circ$ and define fans
$$\Sigma(u):=\{P(\gamma_0^*);\ \gamma_0\preceq \gamma, u\in Q(\gamma_0)^\circ\},\qquad 
\widehat{\Sigma}(u):=\{\gamma_0\preceq\gamma;\ P(\gamma_0)\in\Sigma(u)\}.$$
This gives rise to the following commutative diagram
$$
\xymatrix{
V(q)
\ar@{}[r]|{=}
&
{\bar{X}}
\ar@{}[r]|\subseteq
\ar@{}[d]|{\rotatebox[origin=c]{90}{$\scriptstyle\subseteq$}}
&
{\bar{Z}}
\ar@{}[r]|{=}
\ar@{}[d]|{\rotatebox[origin=c]{90}{$\scriptstyle\subseteq$}}
&
{\ZZ^{n+m}}
\\
&
{\hat{X}} 
\ar@{}[r]|\subseteq
\ar[d]_{\git H}
& 
{\hat{Z}} 
\ar[d]^{\git H}
&
\\
&
X
\ar@{}[r]|\subseteq
&
Z
}
$$ 
where $Z$ and $\widehat{Z}$ are the toric varieties defined by $\Sigma(u)$ and $\widehat{\Sigma}(u)$ respectively, 
$\hat{Z}\rightarrow Z$ is a toric characteristic space for the quasitorus action of $H$ on $\widehat{Z}$ and $\widehat{X}:=\bar{X}\cap\hat{Z}$. 
The resulting variety $X(\nnn,P,u):=X:=\widehat{X}\git H$ is projective, irreducible and normal with dimension, divisor class group and Cox ring
$$\mathrm{dim}(X)=s+r-1,\qquad 
\Cl(X)=K,\qquad 
\RRR(X)=R(\nnn,P).$$
In particular, the variety $X(\nnn,P,u)$ is an intrinsic quadric with Cox ring $R(\nnn,P)$ in homogeneous normal form.
Note that $\widehat{X}\subseteq\overline{X}$ and $\widehat{Z}\subseteq\overline{Z}$ are precisely the sets of $H$-semistable points with respect to the weight $u$.
\end{construction}

\begin{theorem}
Any intrinsic quadric is isomorphic to a variety $X(\nnn,P,u)$ from Construction~\ref{construction:IQ}.
\end{theorem}

By construction any intrinsic quadric $X=X(\nnn,P,u)$ comes embedded inside a toric variety $Z$ defined by a fan $\Sigma(u)$.
We turn to the description of the cones $\sigma\in\Sigma(u)$ defining torus orbits $\TT^{r+s}\cdot z_\sigma$ in $Z$, that intersect $X$ non-trivially:
Let us denote the columns of $P$ by $v_{ij}:=P(e_{ij})$ and $v_k:=P(e_k)$ respectively.
We call a cone $\sigma\in\Sigma(u)$ {\em big (elementary big)}, if its set of primitive ray generators contains for every $i=0,\ldots,r$ at least (precisely) one of the vectors $v_{ij}$.
Moreover, we call $\sigma$ a {\em leaf cone}, if there exists a set of indices $I_\sigma:=\{i_1,\ldots i_{r-1}\}$ 
such that, whenever $v_{ij}$ is a primitive ray generator of $\sigma$, then $i\in I_\sigma$ holds.
Finally, we call a face $\gamma_0\preceq\gamma$ an {\em $X$-face}, 
if the torus orbit $\TT^{n+m}\cdot z_{\gamma_0^*}\subseteq\CC^{n+m}$ defined by the complementary face of $\gamma_0$ intersects $\overline{X}\subseteq\CC^{n+m}$ non-trivially and $P(\gamma_0^*)\in\Sigma(u)$ holds.

\begin{proposition}
Let $X=X(\nnn,P,u)$ be an intrinsic quadric.
Then for any cone $\sigma\in\Sigma(u)$, the following statements are equivalent:
\begin{enumerate}
    \item The torus orbit defined by $\sigma$ intersects $X$ non-trivially.
    \item We have $\sigma=P(\gamma_0^*)$ for an $X$-face $\gamma_0\preceq\gamma$.
    \item The cone $\sigma$ is a big cone or a leaf cone.
\end{enumerate}
\end{proposition}

\begin{remark}
Let $X=X(n,P,u)$ be a $\QQ$-factorial intrinsic quadric with Picard number $\varrho(X)=1$. Then every face $0\neq\gamma_0\preceq\gamma$ that defines a torus orbit $\TT^{n+m}\cdot z_{\gamma_0^*}\subseteq \CC^{n+m}$ intersecting $\overline{X}$ non-trivially is an $X$-face.
In particular, if all entries of $\nnn$ equals one and $m>0$ holds, then we have precisely one elementary big cone in $\Sigma(u)$.
Moreover, if  $\nnn$ contains an index $n_i=2$, we obtain at least two elementary big cones in $\Sigma(u)$.
\end{remark}

We turn to the description of the various cones of divisor classes inside the rational divisor class group of an intrinsic quadric.

\begin{remark}
Let $X=X(\nnn,P,u)$ be an intrinsic quadric. 
Then the cones of effective, movable, semiample and ample divisor classes inside $\Cl(X)_\QQ=K_\QQ$ are given as
$$
\mathrm{Eff}(X)=Q(\gamma),\qquad
\mathrm{Mov}(X)=\bigcap\limits_{{\tiny
\begin{array}{c}
\gamma_0\preceq\gamma\\
\text{facet}
\end{array}}} 
Q(\gamma_0)
$$
$$
\mathrm{SAmple}(X)=\bigcap\limits_{{\tiny
\begin{array}{c}
\gamma_0\preceq\gamma\\ 
X\text{-face}
\end{array}}}
Q(\gamma_0),\qquad
\mathrm{Ample}(X)=\bigcap\limits_{{\tiny
\begin{array}{c}
\gamma_0\preceq\gamma\\ 
X\text{-face}
\end{array}}}
Q(\gamma_0)^\circ.
$$
Note, that due to the projectivity of $X$, the effective cone $\mathrm{Eff}(X)$ is pointed.
\end{remark}

\begin{remark}
Let $X=X(\nnn,P,u)$ be an intrinsic quadric, then $u\in\mathrm{Ample}(X)$ holds. 
Moreover, let $u\not = u'\in\mathrm{Ample}(X)$. 
Then $X(\nnn,P,u)=X(\nnn,P,u')$ holds.
\end{remark}

We turn to the explicit description of the anticanonical divisor class of an intrinsic quadric and connected with it, its Fano property. 
For this, let $R:=R(\nnn,P)$ be a $K$-graded $\CC$-algebra as in Construction~\ref{construction:RnP}. 
We  set $$-\kappa(R):=\sum\mathrm{deg}(T_{ij})+\sum \mathrm{deg}(S_k)-\deg(q)\in K.$$

\begin{proposition}
Let $X:=X(\nnn,P,u)$ be an intrinsic quadric with Cox ring $R:=R(\nnn,P)$. 
Then its anticanonical divisor class is given by $-\KKK_X=-\kappa(R)$.
In particular, if $R:=R(\nnn,P)$ is a $K$-graded $\CC$-algebra as in Construction~\ref{construction:RnP} with $-\kappa(R)\in\mathrm{Mov}(R)^\circ$. 
Then the intrinsic quadric $X(\nnn,P,-\kappa(R))$ is Fano.
\end{proposition}

We turn to singularity types of Fano varieties.
For this, let $X$ be an arbitrary Fano variety and $\pi\colon X'\rightarrow X$ a resolution of singularities, i.e. $\pi$ is  proper and birational and $X'$ is smooth. 
Then, due to the ramification formula, we have 
$$k_{X'} = \pi^*k_X+\sum a_i E_i,$$
where the $E_i$ are prime divisors located in the exceptional locus $\mathrm{Exc}(\pi)$ and the $a_i$ are rational numbers, the so called {\em discrepancies}.
Note that the discrepancies of a Fano variety are independent of the chosen resolution of singularities. 
We call $X$ {\em terminal (canonical, log-terminal)} if all discrepancies $a_i$ are strictly positive (non-negative, strictly greater then $-1$).
Our main tool to characterize these singularity types is the anticanonical complex as introduced in~\cite{BecHauHugNic2016} and developed further in~\cite{HisWro2018}. 
We recall the necessary definitions and results from~\cite{HisWro2018}.
For this, let $X=X(\nnn,P,u)$ be a Fano intrinsic quadric. 
Then, by construction, $X$ is embedded inside a toric variety $Z$. Intersecting $X$ with the open torus $\TT^{r+s}\subseteq Z$, we obtain its {\em tropical variety} as the support of the quasifan
$$\mathrm{trop}(X\cap\TT^{r+s})= |\Sigma_{\PP_r}^{\leq r-1}\times\QQ^s|,$$
where the first factor is the $(r-1)$-skeleton of the standard fan of the $r$-dimensional projective space with primitive ray generators $e_1,\ldots,e_r\in\CC^r$ and $e_0:=-\sum e_i$.
We denote the tropical variety of $X$ with $\mathrm{trop}(X)$ and call its maximal linear subspace the {\em lineality space} $\mathrm{trop}(X)^\mathrm{lin}:=\{0\}\times\QQ^s$.

\begin{construction}
Let $X=X(\nnn,P,u)$ be a Fano intrinsic quadric. 
For every elementary big cone $\sigma=\mathrm{cone}(v_{0j_0},\ldots,v_{rj_r})\in\Sigma(u)$
define numbers
$$
\ell_{\sigma,i} 
\ := \ 
\frac{l_{0j_0} \cdots l_{rj_r}}{l_{ij_i}}
\text{ for } i = 0, \ldots, r
\quad
\text{and}
\quad
\ell_{\sigma}
\ := \ 
 \sum_{i=0}^r \ell_{\sigma, i}-(l_{0j_0} \cdots l_{rj_r}).
$$
\end{construction}

\begin{theorem}
Let $X=X(\nnn,P,u)$ be a Fano intrinsic quadric. Then $X$ is log-terminal if and only if $l_\sigma>0$ holds for all elementary big cones $\sigma\in\Sigma(u)$.
\end{theorem}

\begin{construction}
Let $X=X(\nnn,P,u)$ be a log-terminal Fano intrinsic quadric. For every elementary big cone $\sigma=\mathrm{cone}(v_{0j_0},\ldots,v_{rj_r})\in\Sigma(u)$,
define points inside the lineality space $\mathrm{trop}(X)^\mathrm{lin}$:
$$ 
v_\sigma
\ := \ 
\ell_{\sigma,0} v_{0j_0} + \ldots +  \ell_{\sigma,r} v_{rj_r}
\ \in \ 
\ZZ^{r+s}
\quad
\text{and}
\quad
v_\sigma':=\frac{v_\sigma}{\ell_\sigma}\in\QQ^{r+s}
$$
Then, $v_\sigma'\in\sigma$ holds.
Now, the {\em anticanonical complex} $\AAA$ of $X$ is defined as the polytopal complex obtained as the intersection of the convex hull over the primitive ray generators of $\Sigma(u)$ and the $v_\sigma'$, where $\sigma\in\Sigma(u)$ is elementary big, with the tropical variety $\mathrm{trop}(X)$.
\end{construction}

\begin{remark}
Let $X=X(\nnn,P,u)$ be a log-terminal Fano intrinsic quadric. Then $0\in\AAA^\circ$ holds.
\end{remark}

\begin{theorem}
Let $X(\nnn,P,u)$ be a log-terminal Fano intrinsic quadric.
Then the following holds:
\begin{enumerate}
    \item $X$ is terminal, if and only if the only lattice points of the anticanonical complex are the primitive ray generators of $\Sigma(u)$ and the origin.
    \item $X$ is canonical, if and only if the only interior lattice point of the anticanonical complex is the origin.
\end{enumerate}
\end{theorem}

\section{Proof of Theorem~\ref{theorem:classification}}
\noindent
This section is dedicated to the proof of Theorem~\ref{theorem:classification}.
In a first step, we show that in our situation any intrinsic quadric $X(\nnn,P,u)$ is defined via a trinomial or a quadrinomial relation $q$ in its Cox ring $R(\nnn,P)$. 
Note that the quadrinomial case is part of~\cite{HisWro2018}, where torus actions on singular varieties are investigated, see Remark~\ref{remark:quadrinomials}.
Therefore, we turn to the trinomial case and go through any possible configuration for the defining data $\nnn$ and $P$ to create the classification list. 
Finally, we prove that all of the varieties stated in Theorem~\ref{theorem:classification} are pairwise non-isomorphic.

\begin{lemma}\label{lemma:trinomialsquadrinomials}
Let $X=X(\nnn,P,u)$ be a $\QQ$-factorial Fano intrinsic quadric of dimension three and Picard number one. 
Then $q$ is either a trinomial or a quadrinomial.
\end{lemma}
\begin{proof}
We consider the Cox ring $R(\nnn,P)$ of $X$. 
By assumption we have $n+m=5$ for the number of variables in $R(A,P)$.
Thus, by renaming the variables, we may assume that
$R(A,P)=\CC[T_1,\ldots,T_5]/\langle q\rangle$ holds, where $q$ is a quadratic polynomial contained in the following list:
\begin{enumerate}
    \item $T_1^2$, $T_1T_2$ or $T_1^2+T_2^2$,
    \item $T_1T_2+T_3^2$ or $T_1T_2+T_3T_4$,
    \item any quadratic polynomial with three or four terms,
    \item $T_1^2+T_2^2+T_3^2+T_4^2+T_5^2$.
\end{enumerate}
If $q$ is one of the polynomials in (i), then $R(\nnn,P)$ is not integral; a contradiction.
Now assume $q$ is one of the polynomials in (ii). 
Then the $K_0$-grading on $R(\nnn,P_0)$ turns the total coordinate space $\overline{X}$ into a toric variety and thus $X$ is toric.
This implies, that the Cox ring of $X$ is isomorphic to a polynomial ring; a contradiction to the fact that $\overline{X}$ has a singularity at the origin.
Finally, assume $q=T_1^2+T_2^2+T_3^2+T_4^2+T_5^2$ holds.
Then we obtain 
$$P_0=\left[
\begin{array}{ccccc}
-2&2&0&0&0\\
-2&0&2&0&0\\
-2&0&0&2&0\\
-2&0&0&0&2
\end{array}
\right].$$
Therefore in order to produce a matrix $P$ with primitive columns as in Construction~\ref{construction:RnP}, the matrix $P$ has to be quadratic; a contradiction to $\varrho(X)=1$.
Now, the only case left is $(iii)$ which proves the assertion.
\end{proof}

For the sake of completeness, we extract the quadrinomial case from~\cite{HisWro2018}.

\begin{remark}[Compare~{\cite[Thm. 1.5]{HisWro2018}}]\label{remark:quadrinomials}
Every $\QQ$-factorial Fano intrinsic quadric of dimension three and Picard number one that has at most canonical singularities and a Cox ring $R(\nnn,P)$, where the defining relation $q$ is a quadrinomial, is isomorphic to precisely one of the varieties $X$, specified by its $\Cl(X)$-graded Cox ring $\RRR(X)$, its matrix of generator degrees $Q=[w_1,\ldots,w_r]$ and its anticanonical divisor class $-\KKK_X\in\mathrm{Ample}(X)$ as follows:
{\setlength{\tabcolsep}{10pt}
\begin{longtable}{ccccc}
No.&
$\RRR(X)$&
$\Cl(X)$&
$Q=\left[w_1,\ldots,w_r\right]$&
$-\KKK_X$
\\
\toprule
1&
$
\frac{\CC[T_1,\ldots,T_4,S_1]}{
     \bangle{T_1^2+T_2^2+T_3^2+T_4^2}}
     $
&
$\ZZ\times(\ZZ_2)^3$
&
{\tiny\setlength{\arraycolsep}{2pt}
$\left[\begin{array}{ccccc}
     1&1&1&1&2\\
    \bar 1&\bar 1&\bar 1&\bar 0&\bar 1\\
     \bar 0&\bar 0&\bar 1&\bar 0&\bar 1\\
     \bar 1&\bar 0&\bar 0&\bar 0&\bar 1
\end{array}\right]$
}&
{\tiny\setlength{\arraycolsep}{2pt}
$
\left[
\begin{array}{c}
4\\
\bar 0\\
\bar 0\\
\bar 0
\end{array}
\right]
$
}
\\
\midrule
2&
$
\frac{\CC[T_1,\ldots,T_5]}{
     \bangle{T_1T_2+T_3^2+T_4^2+T_5^2}}$

&
$\ZZ\times(\ZZ_2)^2$
&
{\tiny\setlength{\arraycolsep}{2pt}
$\left[\begin{array}{ccccc}
     1&1&1&1&1  
     \\
     \bar 0&\bar 0&\bar 1&\bar 1&\bar 0
     \\
     \bar 1&\bar 1&\bar 0&\bar 1&\bar 0
\end{array}\right]$
}&
{\tiny\setlength{\arraycolsep}{2pt}
$
\left[
\begin{array}{c}
3
\\
\bar 0
\\
\bar 1
\end{array}
\right]
$
}
\\
\midrule
3&
$
\frac{\CC[T_1,\ldots,T_5]}{\bangle{T_1T_2+T_3^2+T_4^2+T_5^2}}$
&
$\ZZ\times(\ZZ_2)^2$
&
{\tiny\setlength{\arraycolsep}{2pt}
$\left[\begin{array}{ccccc}
     1&3&2&2&2  
     \\
     \bar 1&\bar 1&\bar 1&\bar 1&\bar 0
     \\
     \bar 0 &\bar 0&\bar 1&\bar 0&\bar 1
\end{array}\right]$
}&
{\tiny\setlength{\arraycolsep}{2pt}
$
\left[
\begin{array}{c}
6
\\
\bar0
\\
\bar0
\end{array}
\right]
$
}
\\
\midrule
4&
$
\frac{
     \CC[T_1,\ldots,T_5]}{
     \bangle{T_1T_2+T_3^2+T_4^2+T_5^2}
}$
&
$\ZZ\times\ZZ_2\times\ZZ_6$
&
{\tiny\setlength{\arraycolsep}{2pt}
$\left[\begin{array}{ccccc}
     1&1&1&1&1  
     \\
     \bar 1&\bar 1&\bar 1&\bar 0&\bar 0
     \\
     \bar 2&\bar 4&\bar 3&\bar 3&\bar 0
\end{array}\right]$
}&
{\tiny\setlength{\arraycolsep}{2pt}
$
\left[
\begin{array}{c}
3
\\
1
\\
0
\end{array}
\right]
$
}
\\
\bottomrule
\end{longtable}
}
\noindent
Note, that these varieties appear as Nos.~64, 69, 77 and~79 in Theorem~\ref{theorem:classification}.
\end{remark}

Let us turn to the trinomial case. In a first step, we list the possible choices of the data $\nnn$ and $m$. 
Then we proceed with Settings~\ref{setting:trinom1}, \ref{setting:trinom2} and~\ref{setting:trinomial3} by investigating these cases to finally obtain in Remark~\ref{remark:trinom1Final} and Propositions~\ref{proposition:trinom2Final} and~\ref{proposition:trinomial3Final} 
the finitely many possible choices for the matrix $P$.

\begin{remark}
Let $X:=X(\nnn,P,u)$ be a $\QQ$-factorial intrinsic quadric of dimension three and Picard number one with Cox ring $R(\nnn,P)$. 
Then we have $n+m = 5$ for the defining data $n$ and $m$. 
In particular, if the defining relation $q$ is a trinomial, we obtain $r=s=2$ due to the Picard number of $X$ and we are in one of the following situations.
\begin{enumerate}
    \item $\nnn=(1,1,1)$\quad and\quad $m=2$.
    \item $\nnn=(2,1,1)$\quad and\quad $m=1$.
    \item $\nnn=(2,2,1)$\quad and\quad $m=0$.
\end{enumerate}
\end{remark}

\begin{remark}
Let $R(\nnn,P)$ be a $K$-graded $\CC$-algebra as in Construction~\ref{construction:RnP}. 
We call the following {\em admissible operations} on $P$:
\begin{enumerate}
    \item Add a multiple of one of the first $r$-rows to one of the last $s$-rows.
    \item Any elementary row operation between the last $s$-rows.
    \item Swap two columns $v_{i1}$ and $v_{i2}$.
    \item Swap two columns of the last $m$ columns.
\end{enumerate}
The operations of type i) and ii) does not effect the ring $R(\nnn,P)$. 
Types iii) and iv) leaves the graded isomorphy type of $R(\nnn,P)$ invariant.
\end{remark}

\begin{setting}\label{setting:trinom1}
Let $X:=X(\nnn,P,u)$ be a $\QQ$-factorial Fano intrinsic quadric of dimension three and Picard number one, having at most canonical singularities and Cox ring $R(\nnn,P)$ with $\nnn=(1,1,1)$ and $m=2$. 
Then, by construction, the matrix $P$ is an integral $(4\times 5)$-matrix of the following form:
$$
P=\left[
\begin{array}{ccccc}
-2&2&0&0&0\\
-2&0&2&0&0\\
x_1&x_2&x_3&x_4&x_5\\
y_1&y_2&y_3&y_4&y_5
\end{array}
\right]
$$
\end{setting}

\begin{remark}\label{remark:trinom1LatticePoly}
Situation as in~\ref{setting:trinom1}.
As $X$ has Picard number one and by the definition of $\Sigma(u)$, we obtain a big cone and an associated vertex of the anticanonical complex of $X$: 
$$\sigma=\mathrm{cone}(v_{01},v_{11},v_{21})\in\Sigma,\qquad v_\sigma'=[0,0,x_1+x_2+x_3,y_1+y_2+y_3].$$
In particular, forgetting about the first two coordinates, the anticanonical complex of $X$ intersected with the lineality space $\mathrm{trop}(X)^\mathrm{lin}$ is the two-dimensional lattice polytope
$$
\Delta:=\mathrm{conv}([x_4,y_4], [x_5,y_5],[z_1,z_2]),\qquad [z_1,z_2]:=[x_1+x_2+x_3,y_1+y_2+y_3].
$$
As $X$ has at most canonical singularities, the origin is the only interior lattice point of $\Delta$.
Thus, by applying admissible operations on the last two rows of $P$, we may assume that $\Delta$ is one of the $16$ two-dimensional reflexive polytopes~\cite{Bat1985,Koe1991,Rab1989}. 
In particular, as $\Delta$ has three vertices, we may assume that it is one of the following:
{\small
$$
\mathrm{conv}([1,0],[0,1],[-1,-1]),\ 
\mathrm{conv}([1,1],[-1,1],[0,-1]),\ 
\mathrm{conv}([1,1],[-1,1],[-1,-2]),
$$
$$
\mathrm{conv}([1,1],[-1,1],[-1,-3]),\ 
\mathrm{conv}([2,1],[-1,1],[-1,-2]).
$$
}
\end{remark}

\begin{remark}\label{remark:trinom1Final}
Situation as in~\ref{remark:trinom1LatticePoly}. Then, the vertices of $\Delta$ are invariant under adding a multiple of the first two rows of $P$ to one of the last two rows of $P$. Thus we may assume in addition, that we have $x_2,x_3,y_2,y_3\in\{0,1\}$. 
Note that any such choice fixes all entries of $P$, due to the definition of the vertex $[z_1,z_2]$.
Thus, in this situation we have only finitely many possibilities for the matrix $P$ to check.
\end{remark}

\begin{remark}
Let $X=X(\nnn,P,u)$ be a $\QQ$-factorial intrinsic quadric of Picard number one
and consider the rational degree-vector
$$d:=(\deg_\QQ(T_{01}),\ldots,\deg_\QQ(T_{rn_r}),\deg_\QQ(T_1),\ldots,\deg_\QQ(T_m)).$$
By construction, we have $\mathrm{Lin}_\QQ(d) =  \mathrm{ker}_\QQ(P)$ and as the effective cone of $X$ is pointed, we may assume $d \in \QQ_{>0}^{n+m}.$
Now, 
denote with
$P_{ij}$ resp. $P_k$
the submatrices of $P$ arising by deleting the $ij$-th resp. $k$-th column
and set $w_{ij}:=\mathrm{det}(P_{ij}), \ w_k:=\mathrm{det}(P_k)$.
Then we obtain a non-zero vector
$$
(w_{01},\ldots,w_{rn_r},w_1, \ldots w_m)\in\mathrm{ker}_\QQ(P).
$$
In particular, the $w_{ij}$ and $w_k$ are either all positive or negative. We call them the {\em rational weights}.
\end{remark}

\begin{setting}\label{setting:trinom2}
Let $X:=X(\nnn,P,u)$ be a $\QQ$-factorial Fano intrinsic quadric of dimension three, Picard number one, having at most canonical singularities and Cox ring $R(\nnn,P)$ with $\nnn=(2,1,1)$ and $m=1$. 
By applying admissible operations we may assume to be in the following situation:
$$
\begin{array}{cc}
P=\left[
\begin{array}{ccccc}
-1&-1&2&0&0\\
-1&-1&0&2&0\\
0&x_2&x_3&x_4&x_5\\
0&0&y_3&y_4&y_5
\end{array}
\right],\qquad
&
\begin{array}{c}
x_2 > 0,\\
0< x_5\leq|y_5|,\\
x_4,y_4\in\{0,1\}.
\end{array}
\end{array}
$$
Moreover, by multiplying the last row with $(-1)$, if necessary, we may assume that we have positive weights:
\begin{align*}
w_{01} &= 4x_2y_5+2x_3y_5-2x_5y_3+2x_4y_5-2x_5y_4,\\
w_{02} &= -2x_3y_5+2x_5y_3-2x_4y_5+2x_5y_4,\\
w_{11} &= 2x_2y_5,\\
w_{21} &= 2x_2y_5,\\
w_{1} &= -2x_2y_3-2x_2y_4.
\end{align*}
Note that the last row operation possibly changes the sign of $y_4$. Thus we may only assume that $y_4\in\{-1,0,1\}$ holds.
\end{setting}

\begin{remark}\label{remark:trinom2Delta}
Situation as in~\ref{setting:trinom2}.
As $X$ has Picard number one, we obtain two big cones with associated vertices of the anticanonical complex of $X$:
$$\renewcommand{\arraystretch}{1.5}
\begin{array}{rcl}
\sigma_1&=&\mathrm{cone}(v_{01},v_{11},v_{21}),\\
\sigma_2&=&\mathrm{cone}(v_{02},v_{11},v_{21}),
\end{array}
\qquad
\begin{array}{rcl}
v_{\sigma_1}'&=&[0,0,\frac{1}{2}(x_3+x_4),\frac{1}{2}(y_3+y_4)])\\
v_{\sigma_2}'&=&[0,0,\frac{1}{2}(x_3+x_4)+x_2,\frac{1}{2}(y_3+y_4)].
\end{array}
$$
In particular, forgetting about the first coordinates, the anticanonical complex of $X$ intersected with the lineality space $\mathrm{trop}(X)^\mathrm{lin}$ is a triangle $\Delta = \mathrm{conv}(p_1,p_2,p_3)$ with
$$
\begin{array}{c}
p_1=[\frac{1}{2}(x_3+x_4),\frac{1}{2}(y_3+y_4)],\
p_2=[\frac{1}{2}(x_3+x_4)+x_2,\frac{1}{2}(y_3+y_4)],\
p_3=[x_5,y_5].
\end{array}
$$
\end{remark}

\begin{remark}\label{remark:trinom2Delta2}
Situation as in Remark~\ref{remark:trinom2Delta}.
We investigate the polytope $\Delta$.
First note, that by assumption $x_5>0$ holds and we obtain $y_5>0$, as $x_2$ and $w_{11}=2x_2y_5$ are positive.
In particular, the vertex $p_3$ is contained in the positive orthant.
Moreover, as $w_1$ is positive, we conclude $y_3+y_4<0$ and thus the points $p_1$ and $p_2$ are contained in the lower half plane.
Note that the line segment $\overline{p_1p_2}$ is parallel to the $x$-axis. As $X$ is Fano, we have $0\in\Delta^\circ$ and conclude $x_3+x_4<0$, as $x_2$ is positive.
We sketch the situation:
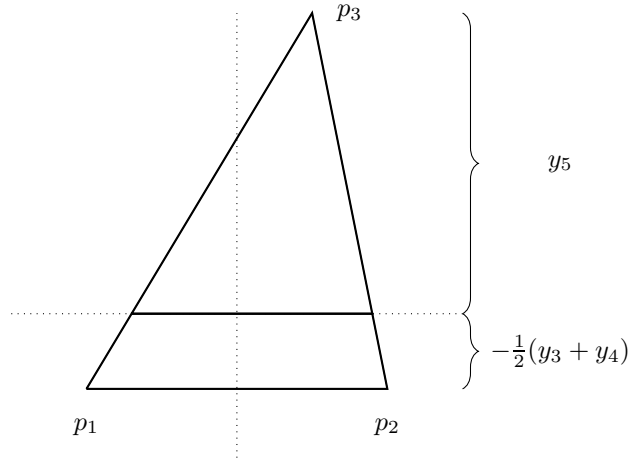
\begin{center}
\begin{tikzpicture}
    \draw[dotted] (-3,0) -- (3,0);
    \draw[dotted] (0,-2) -- (0,4);
    \draw (-2,-1)--(2,-1)--(1,4)--(-2,-1);
    \draw[thick] (-2,-1)--(2,-1)--(1.8,0)--(-1.4,0)--(-2,-1);
    \draw[thick]    (-1.4,0)--(1,4)--(1.8,0)--(-1.4,0);
    \node (p1) at (-2,-1.5) {$p_1$};
    \node (p2) at (2,-1.5) {$p_2$};
    \node (p3) at (1.5,4) {$p_3$};
    \draw[decorate,decoration={brace,amplitude=6pt}] (3,4)--(3,0);
    \draw[decorate,decoration={brace,amplitude=6pt}] (3,0)--(3,-1);
    \node (h1) at (4.3,-0.5) {$-\frac{1}{2}(y_3+y_4)$};
    \node (h2) at (4.3,2) {$y_5$};
\end{tikzpicture}
\end{center}
Note, that in this situation we can not determine the position of $p_2$ with respect to the $y$-axis. 
\end{remark}

\begin{proposition}\label{proposition:trinom2Final}
Situation as in Setting~\ref{setting:trinom2}.
Then we obtain the following estimates for the entries of $P$:
$$
0<x_2\leq 2-(y_3+y_4),\qquad
0\leq x_4\leq 1,\qquad
0<x_5\leq|y_5|,
$$
$$
-18-y_4\leq y_3<-y_4,\qquad
-1\leq y_4\leq 1
$$
$$
\begin{array}{c}
2x_5y_3+2x_5y_4-2x_4y_5-4x_2y_5\\
\hline
2y_5
\end{array}<x_3<-x_4,\quad
0<y_5\leq\begin{cases}
9&x_2=1\\
\frac{x_2-\frac{1}{2}(y_3+y_4)}{x_2-1}&\text{else.}
\end{cases}
$$
\end{proposition}
\begin{proof}
Note that by assumption $x_2>0$, $x_4\in\{0,1\}$, $y_4\in\{-1,0,1\}$ and $0 < x_5 \leq |y_5|$ hold.
Now, positivity of the weights $w_{01}$, $w_{11}$ and $w_1$ imply
$$
\begin{array}{c}
2x_5y_3+2x_5y_4-2x_4y_5-4x_2y_5\\
\hline
2y_5
\end{array}<x_3,\qquad
0<y_5\qquad \text{and}\qquad
y_3<-y_4.
$$
Moreover, similar as in Remark~\ref{remark:trinom2Delta2}, we have $\frac{1}{2}(x_3+x_4)<0$ and conclude $x_3 < -x_4$.
We investigate slices of the polytope $\Delta$:
Due to the singularity type of $X$, we have 
$$
|\Delta\cap\{y=0\}| = x_2+x_2\frac{\frac{1}{2}(y_3+y_4)}{y_5-\frac{1}{2}(y_3+y_4)} \leq 2.
$$
Thus, reordering suitably and using $y_5> 0$ yields
$$x_2\leq
2-\frac{(y_3+y_4)}{y_5}\leq 2-(y_3+y_4).$$
Similarly, we have 
$$\Delta\cap\{y=1\}=x_2-x_2\frac{1+\frac{1}{2}(y_3+y_4)}{y_5-\frac{1}{2}(y_3+y_4)}\leq 1.$$
In particular, if $x_2\not = 1$ holds, this implies
$$y_5\leq \frac{x_2-\frac{1}{2}(y_3+y_4)}{x_2-1}.$$
We proceed by investigating the tetrahedron $\Delta'$ defined by the following vertices:
$$
\begin{array}{c}
[0,0,x_5,y_5],\qquad [-1,-1,x_2,0],
\end{array}
$$
$$
\begin{array}{c}
[0,0,\frac{1}{2}(x_3+x_4),\frac{1}{2}(y_3+y_4)],\qquad
[0,0,\frac{1}{2}(x_3+x_4)+x_2,\frac{1}{2}(y_3+y_4)].
\end{array}
$$
Note that by construction $\Delta'$ is contained in the anticanonical complex of $X$ and thus has the origin as its unique interior lattice point. 
The polytope $\Delta'$ is living inside the linear space spanned by $[1,1,0,0],[0,0,1,0]$ and $[0,0,0,1]$.
In particular, we may regard $\Delta'$ as a polytope in $\QQ^3$ by forgetting about the first coordinate. 
Now, $\Delta'$ is contained in the lattice polytope $\Delta''$ defined by the following vertices:
$$[-1,x_2,0],\quad[1,x_3+x_4-x_2],\quad
[1, x_3+x_4+x_2,y_3+y_4],\quad[1,2x_5-x_2,2y_5].
$$
Note that by construction $\Delta''$ is a lattice polytope having the origin as its unique interior lattice point. Thus, due to~\cite[Thm 2.2]{AveKreNil2015}, its standard $\QQ^3$-volume is bounded by $12$ which gives
\begin{equation}\label{equation:equ1}
  \frac{2}{3}x_3(y_3+y_4)-\frac{4}{3}x_3y_5\leq 12  
\end{equation}
Now, reordering yields
$$\frac{18}{x_3}+2y_5-y_4\leq y_3$$
and as $\frac{1}{x_3}\geq -1$ and $y_5>0$ hold, we obtain at $-18-y_4\leq y_3.$
Moreover, reordering Equation~\ref{equation:equ1} once more, we arrive at
$$-\frac{4}{3}x_3y_5\leq12-\frac{2}{3}x_3(y_3+y_4).$$
Using positivity of $x_3(y_3+y_4)$ and $-1\leq \frac{1}{x_3}< 0$, we conclude $y_5\leq 9$.
\end{proof}

\begin{setting}\label{setting:trinomial3}
Let $X:=X(\nnn,P,u)$ be a $\QQ$-factorial Fano intrinsic quadric of dimension three, Picard number one, having at most canonical singularities and Cox ring $R(\nnn,P)$ with $\nnn=(2,2,1)$ and $m=0$. Then, by applying admissible operations on $P$, we may assume to be in the following situation:
$$
\begin{array}{cc}
P=\left[
\begin{array}{ccccc}
-1&-1&1&1&0\\
-1&-1&0&0&2\\
0&x_2&x_3&0&x_5\\
0&0&y_3&0&y_5
\end{array}
\right],\qquad
&
\begin{array}{c}
x_2 > 0,\\
0< x_5\leq|y_5|.
\end{array}
\end{array}
$$
Moreover, by multiplying the last row with $(-1)$, if necessary, we may assume, that we have positive weights:
\begin{align*}
w_{01} &= 2x_2y_3-x_3y_5+x_5y_3,\\
w_{02} &= x_3y_5-x_5y_3,\\
w_{11} &= -x_2y_5,\\
w_{12} &= 2x_2y_3 + x_2y_5,\\
w_{21} &= x_2y_3.
\end{align*}
\end{setting}

\begin{remark}\label{remark:trinomial3Delta}
Situation as in~\ref{setting:trinomial3}.
As $X$ has Picard number one, we obtain four big cones with associated vertices of the anticanonical complex
$$\renewcommand{\arraystretch}{1.5}
\begin{array}{rcl}
\sigma_1&=&\mathrm{cone}(v_{01},v_{11},v_{21}),\\
\sigma_2&=&\mathrm{cone}(v_{01},v_{12},v_{21}),\\
\sigma_3&=&\mathrm{cone}(v_{02},v_{11},v_{21}),\\
\sigma_4&=&\mathrm{cone}(v_{02},v_{12},v_{21}),
\end{array}
\qquad
\begin{array}{rcl}
v_{\sigma_1}'&=&[0,0,\frac{1}{3}x_5+\frac{2}{3}x_3,\frac{1}{3}y_5+\frac{2}{3}y_3],\\
v_{\sigma_2}'&=&[0,0,\frac{1}{3}x_5,\frac{1}{3}y_5],\\
v_{\sigma_3}'&=&[0,0,\frac{1}{3}x_5+\frac{2}{3}x_3+\frac{2}{3}x_2,\frac{1}{3}y_5+\frac{2}{3}y_3],\\
v_{\sigma_4}'&=&[0,0,\frac{1}{3}x_5+\frac{2}{3}x_2,\frac{1}{3}y_5].
\end{array}
$$
In particular, forgetting about the first coordinates, the anticanonical complex of $X$ intersected with the lineality space $\mathrm{trop}(X)^\mathrm{lin}$ is a trapezoid $\Delta = \mathrm{conv}(p_1,p_2,p_3,p_4)$, with
$$
\renewcommand{\arraystretch}{1.5}
\begin{array}{rcl}
p_1&=&[\frac{1}{3}x_5+\frac{2}{3}x_3,\frac{1}{3}y_5+\frac{2}{3}y_3],\\
p_3&=&[\frac{1}{3}x_5+\frac{2}{3}x_3+\frac{2}{3}x_2,\frac{1}{3}y_5+\frac{2}{3}y_3],
\end{array}
\qquad
\begin{array}{rcl}
p_2&=&[\frac{1}{3}x_5,\frac{1}{3}y_5],\\
p_4&=&[\frac{1}{3}x_5+\frac{2}{3}x_2,\frac{1}{3}y_5].
\end{array}
$$
\end{remark}

\begin{remark}\label{remark:trinomial3Delta2}
Situation as in~\ref{remark:trinomial3Delta}.
We investigate the polytope $\Delta$. 
In a first step, we determine the position of its vertices relative to the $x$- and $y$-axis.
First note that by assumption $x_2$ and $x_5$ are positive and by positivity of the weight $w_{11}=-x_2y_5$ we obtain $y_5<0$.
In particular, we have $(p_2)_1, (p_4)_1> 0$ and $(p_2)_2, (p_4)_2< 0$. 
Moreover, as $X$ is Fano, we obtain $0\in\Delta^\circ$ and thus $(p_1)_1<0$ and $(p_1)_2>0$ due to the positivity of $x_2$. 
We sketch the situation:
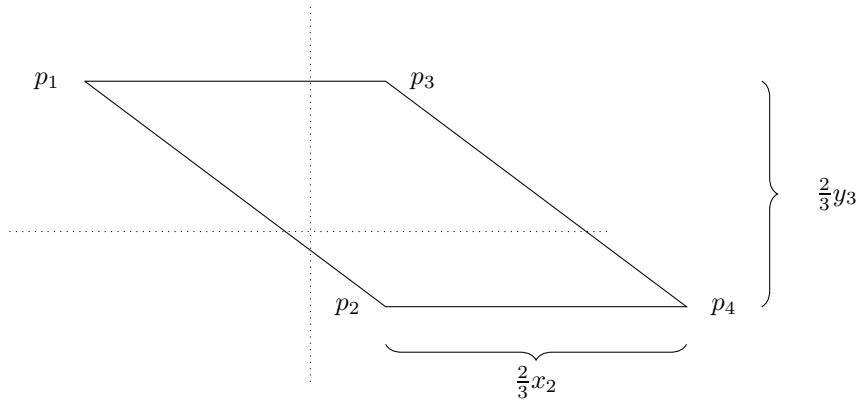
\begin{center}
\begin{tikzpicture}
    \draw[dotted] (-4,0) -- (4,0);
    \draw[dotted] (0,-2) -- (0,3);
    \draw (-3,2)--(1,2)--(5,-1)--(1,-1)--(-3,2);
    \node (p1) at (-3.5,2) {$p_1$};
    \node (p2) at (0.5,-1) {$p_2$};
    \node (p3) at (1.5,2) {$p_3$};
    \node (p4) at (5.5,-1) {$p_4$};
    \draw[decorate,decoration={brace,amplitude=6pt}] (6,2)--(6,-1);
    \draw[decorate,decoration={brace,amplitude=6pt}] (5,-1.5)--(1,-1.5);
    \node (h) at (7,0.5) {$\frac{2}{3}y_3$};
    \node (b) at (3,-2) {$\frac{2}{3}x_2$};
\end{tikzpicture}
\end{center}
Note that we can not determine the position of $p_3$ with respect to the $y$-axis. \end{remark}

\begin{proposition}\label{proposition:trinomial3Final}
Situation as in~\ref{setting:trinomial3}.
Then we obtain the following estimates for the entries of $P$:
$$
0<x_2\leq3,\qquad
\frac{2x_2y_3+x_5y_3}{y_5}<x_3<\frac{x_5y_3}{y_5},\qquad
0<x_5\leq|y_5|,
$$
$$
0<y_3\leq \frac{72}{3(x_2+1)},\qquad
-2y_3<y_5<0.
$$
\end{proposition}
\begin{proof}
Note that by assumption $x_2>0$ and $0< x_5 \leq |y_5|$ holds.
Now, positivity of the weights $w_{21}, w_{11}$ and $w_{12}$ imply
$$
y_3>0,\qquad y_5<0 \qquad \text{and} \qquad -2y_3<y_5.
$$
Thus, using positivity of $w_{01}$ and $w_{02}$, we conclude
$$
x_3>\frac{2x_2y_3+x_5y_3}{y_5}\qquad \text{and}\qquad x_3<\frac{x_5y_3}{y_5}.
$$
Now, due to the singularity type of $X$, the $y=0$ slice of $\Delta$ implies
$$
|\Delta\cap\{y=0\}| = \frac{2}{3}x_2 \leq 2
$$
and thus $x_2 \leq 3$ holds. 
We proceed by investigating the pyramid
$$
\Delta':=\mathrm{conv}([0,2,x_5,y_5],v_{\sigma_1}',\ldots,v_{\sigma_4}'),
$$
By construction $\Delta'$ is contained in the anticanonical complex of $X$
and by deleting the first coordinate, we may regard $\Delta'$ as a polytope inside $\QQ^3$ having the origin as its unique interior lattice point, due to the singularity type of $X$.
We proceed by modifying $\Delta'\subseteq\QQ^3$. 
By extending the edges starting in $[2,x_5,y_5]$, we enlarge $\Delta'$ to the lattice polytope $\Delta''$ having the following vertices:
$$
[2,x_5,y_5],\qquad [-1,x_3,y_3],\qquad [-1,0,0],\qquad [-1,x_2+x_3,y_3],\qquad [-1,x_2,0].
$$
Note that by construction $\Delta''$ still has the origin as its unique interior lattice point. Thus, due to~\cite[Thm. 2.2]{AveKreNil2015}, its standard $\QQ^3$-volume is bounded by $12$ and we conclude
$$
y_3\leq\frac{72}{3(x_2+1)}.
$$
\end{proof}

\begin{proof}[Proof of Theorem~\ref{theorem:classification} (the classification list)]
Due to Lemma~\ref{lemma:trinomialsquadrinomials} and Remark~\ref{remark:quadrinomials} we only need to consider the trinomial case.
Then, due to Remark~\ref{remark:trinom1Final} and Propositions~\ref{proposition:trinom2Final} and~\ref{proposition:trinomial3Final} we only have finitely many possible Fano varieties $X(\nnn,P,u)$ to check.
Computing the anticanonical complex for all possible configurations the resulting canonical Fano varieties are listed in Theorem~\ref{theorem:classification} with Nos.~1 - 63, 65 - 68, 70 - 76 and~78.
The missing varieties are directly imported from~\cite[Thm. 1.5]{HisWro2018} as Nos.~64, 69, 77 and~79.
\end{proof}

Now we turn to the irredundancy of the classification list.

\begin{remark}\label{remark:invariants}
Let $X=X(\nnn,P,u)$ be an $n$-dimensional intrinsic quadric. 
Then the following numbers are invariants of $X$:
\begin{enumerate}
    \item The {\em anticanonical self-intersection number} $-\KKK_X^n$, which can be directly computed via~\cite[Constr. 3.3.3.4]{ArzDerHauLaf2015}.
    \item The {\em Fano index} $q(X)$, which is defined as the largest integer $q(X)$, such that $-\KKK_X = q(X)\cdot w$ holds with some $w \in \Cl(X)$. 
    \item The {\em Picard index} $p(X)$, which is defined as the index of the Picard group inside the divisor class group. 
    Note, that in our situation, the Picard group is given as
    $$
    \mathrm{Pic}(X)=\bigcap\limits_{{\tiny
    \begin{array}{c}
    \gamma_0\preceq\gamma\\ 
    X\text{-face}
    \end{array}}}
    Q(\gamma_0\cap\ZZ^{n+m})\subseteq\Cl(X).
    $$
    \item The dimension of the automorphism group $\dim(\mathrm{Aut(X)})$.
\end{enumerate}
Moreover, if $X$ is isomorphic to another intrinsic quadric $X'=X(\nnn',P',u')$, then $\RRR(X)$ and $\RRR(X')$ are isomorphic as graded rings. In this case, the following holds:
\begin{enumerate}
    \item We have $\dim(\overline{X}^\mathrm{\ sing})=\dim(\overline{X'}^\mathrm{\ sing})$.
    \item There is a bijection between the set of generator degrees $\Omega_X$ and $\Omega_{X'}$.
    \item The sets $\Omega^{\dim}_X:=\{ \dim(\RRR(X)_w);\  w\in\Omega_X\}$  and $\Omega^{\dim}_{X'}$ coincide.
\end{enumerate}
\end{remark}

\begin{proposition}
The varieties defined by the data in Theorem~\ref{theorem:classification} are pairwise non-isomorphic.
\end{proposition}
\begin{proof}
We denote by $X_i$ the Fano variety defined by the $i$-th datum in Theorem~\ref{theorem:classification}, by $\RRR_i$ its Cox ring, by $\overline{X_i}$ its total coordinate space and by
$\Omega_i=\{w_1,\ldots,w_r\}$ its set of generator degrees. 
As the divisor class group, the Fano index and the anticanonical self-intersection number presented in Theorem~\ref{theorem:classification} are invariants, we only need to compare those varieties $X_i$ and $X_j$, where all these data coincide. 
The next table presents invariants of these varieties, where the cases to compare are divided via horizontal lines:
{\setlength{\tabcolsep}{5pt}
\begin{longtable}{c|ccc}
$i$
&
$p(X_i)$
&
$\dim(\mathrm{Aut(X_i)})$
&
$\dim(\overline{X_i}^\mathrm{\ sing})$
\\
\toprule
16
&
24
&
2
&
1
\\
17
&
24
&
2
&
0
\\
\midrule
20
&
48
&
2
&
1
\\
21
&
24
&
2
&
1
\\
\midrule
27
&
240
&
2
&
1
\\
28
&
120
&
2
&
1
\\
\midrule
30
&
24
&
2
&
1
\\
31
&
48
&
2
&
1
\\
\midrule
33
&
9
&
2
&
0
\\
34
&
9
&
2
&
0
\\
\midrule
35
&
54
&
2
&
0
\\
36
&
18
&
2
&
0
\\
\midrule
41
&
16
&
2
&
0
\\
42
&
16
&
2
&
1
\\
43
&
8
&
2
&
1
\\
\midrule
46
&
48
&
2
&
1
\\
47
&
48
&
2
&
1
\\
\midrule
50
&
36
&
2
&
0
\\
51
&
36
&
2
&
0
\\
52
&
36
&
2
&
1
\\
\midrule
55
&
64
&
2
&
1
\\
56
&
64
&
2
&
1
\\
\midrule
62
&
8
&
2
&
2
\\
63
&
8
&
2
&
1
\\
64
&
8
&
1
&
0
\\
\midrule
67
&
48
&
2
&
2
\\
68
&
48
&
2
&
1
\\
69
&
48
&
1
&
0
\\
\midrule
75
&
72
&
2
&
2
\\
76
&
72
&
2
&
1
\\
77
&
72
&
1
&
0
\\
\bottomrule
\end{longtable}
}
\noindent 
There are only 4 cases left, that can not be distinguished via the table above. We treat them in the following paragraphs:

\vspace{2pt}\noindent{\em $X_{33}$ and $X_{34}$.} 
In this case the homogeneous component of $\RRR_{33}$ of degree $(1,0)\in\Omega_{33}$ has dimension three. This is in contrast to to $R_{34}$, where the maximal dimension of the homogeneous components with respect to the generator degrees in $\Omega_{34}$ is two.

\vspace{2pt}\noindent{\em $X_{46}$ and $X_{47}$.} 
In this situation, all homogeneous components of $\RRR_{46}$ with respect to the weights in $\Omega_{46}$ are one-dimensional which is in contrast to the two-dimensional homogeneous component of $\RRR_{47}$ of degree $(2,0)\in\Omega_{47}$.

\vspace{2pt}\noindent{\em $X_{50}$ and $X_{51}$.} 
Note, that due to Remark~\ref{remark:invariants}, we have a bijection $\Omega_{50}\rightarrow \Omega_{51}$. 
Now $|\Omega_{50}|=5$ which is in contrast to $|\Omega_{51}|=4$.

\vspace{2pt}\noindent{\em $X_{55}$ and $X_{56}$.}
Assume there is a graded isomorphism $\RRR_{55}\rightarrow\RRR_{56}$.
Then we have an isomorphism $\Cl(X_{55})\rightarrow\Cl(X_{56})$ mapping $\Omega_{55}$ onto $\Omega_{56}$. 
We go through the possible images of $(1,\bar 1)\in\Omega_{55}$: Assume that $(1,\bar 1)$ is mapped on either $(1,\bar 1)$ or $(1,\bar 5)$. Then $(2,\bar 2)\in\Omega_{55}$ is mapped on $(2,\bar 2)$ which is not contained in $\Omega_{56}$; a contradiction. Now assume $(1,\bar 1)$ is mapped on $(1,\bar 0)$ or $(1,\bar 2)$ then $(2,\bar 2)$ is mapped on either $(2,\bar 0)$ or $(2,\bar 4)$ which are not contained in $\Omega_{56}$; a contradiction; Finally assume that $(1,\bar 1)$ is mapped on $(2,\bar 6)$. Then $(2,\bar 2)$ is mapped on $(4,\bar 4)$ which is again not contained in $\Omega_{56}$; a contradiction. 
This implies that there is no graded isomorphism $\RRR_{55}\rightarrow\RRR_{56}$ and thus $X_{55}$ and $X_{56}$ can not be isomorphic.
\end{proof}

\bibliographystyle{plain}
\bibliography{bibliography}{}
\end{document}